\documentclass[psamsfonts, reqno]{amsart}
\usepackage{amssymb,amsfonts}
\usepackage{enumerate}
\usepackage{mathrsfs}
\usepackage{graphicx}
\usepackage{caption}
\numberwithin{figure}{section}

\newtheorem{thm}{Theorem}[section]

\newtheorem{prop}[thm]{Proposition}
\newtheorem{lem}[thm]{Lemma}

\theoremstyle{definition}
\newtheorem{defn}[thm]{Definition}

\theoremstyle{remark}
\newtheorem{rem}[thm]{Remark}

\newcommand{\sF}{{\mathcal F}}
\newcommand{\sG}{{\mathcal G}}

\newcommand{\uu}{{u}}

\newcommand{\ZZ}{{\mathbb Z}}
\newcommand{\W}{{\nu}}

\numberwithin{equation}{section}
\newcommand{\f}{{f}}
\newcommand{\F}{\overline{F}}
\newcommand{\G}{\overline{G}}
\newcommand{\ord}{{\rm ord}}
\newcommand{\dgt}{{d}}
\newcommand{\Sum}{{S}}

\newcommand{\SP}{{\rm S}}

\title{Products of   Binomial Coefficients and  Unreduced Farey Fractions}

\author{Jeffrey C. Lagarias}
\address{Department of Mathematics, University of Michigan,
Ann Arbor, MI 48109-1043, USA}
\email{lagarias@umich.edu}

\author{Harsh Mehta}
\address{Department of Mathematics, University of South Carolina,
Columbia, SC 29208, USA}
\email{hmehta@math.sc.edu}

\subjclass[2010]{Primary: 11B65, Secondary:  05A10, 11B57, 11N05, 11N64}

\thanks{Work of the first author was supported by NSF Grants DMS-1101373 and DMS-1401224.}

\begin{document}

\begin{abstract}
This paper studies the product $\G_n$ 
of the binomial coefficients in the $n$-th row of Pascal's triangle,
which equals the reciprocal of  the product 
of all the reduced and unreduced Farey fractions of order $n$. 
It studies its size as a real number, measured by  $\log(\G_n)$, and its prime factorization,
measured by the order of divisibility  $\W_p(\G_n) =\ord_p(\G_n)$  by a fixed prime $p$, each viewed
as a function of $n$. It derives three formulas for 
$\ord_p(\G_n)$, two of which relate it to base $p$ radix expansions of integers up to $n$,
and which  display different facets of its behavior. 
These formulas are used to determine the maximal growth rate of each
$\ord_p(\G_n)$ and to explain structure of  
the fluctuations of these functions. 
It also defines  analogous functions $\W_b(\G_n)$ for all integer bases $b \ge 2$
using base $b$ radix expansions replacing base $p$-expansions. 
A final topic relates  factorizations of $\G_n$ to Chebyshev-type prime-counting estimates
and the prime number theorem.

\end{abstract}

\maketitle

%
%
%
\section{Introduction}\label{sec1}

The complete products of binomial coefficients of order $n$ are the integers
$$
\G_n:= \prod_{k=0}^n {{n}\choose{k}}.
$$
This integer sequence begins $\G_1=1, \,\G_2=2, \,\G_3=9, \, \G_4=96, \,\G_5=2500,  \G_6=162000,$ and   $\G_7=26471025,$ 
and appears as  
A001142 in OEIS \cite{OEIS}.
The integer $\G_n$ is the reciprocal of 
 the product $G_n$ of all nonzero unreduced Farey fractions of order $n$,
  as we describe in Section \ref{sec2a}.
We encountered  unreduced Farey products  $G_n$ while investigating  
the  products $F_n$ of all nonzero (reduced) Farey fractions.
 The connections with Farey fractions and their relations   to prime number theory  motivated this work.

We study the size of the integers  $\G_n$ viewed as real numbers
and   the behavior
of their prime factorizations, as functions of $n$.
Since  the $\G_n$  grow exponentially fast we measure their size in terms of  
the rescaled function
\begin{equation}
\W_{ \infty}(\G_n) := \log(\G_n).
\end{equation}
It is easy to  show that  $\log(\G_n)$ has  smooth growth,
given by  an asymptotic expansion having leading term $\frac{1}{2} n^2$. 
We derive the first few terms of its asymptotic expansion in Section \ref{sec2b},
which are obtainable using Stirling's formula.  We observe that 
from the Farey fraction viewpoint this asymptotic estimate 
 has an analogy with a formulation of the Riemann
hypothesis for Farey fractions due to Mikol\'{a}s \cite{Mik51}.
The function  $\log(\G_n)$ actually
has a complete  asymptotic expansion in negative powers $\frac{1}{n^k}$ valid to all orders
after its first few lead terms.  This full asymptotic expansion  is derived in  Appendix A,
in which we make use of  known asymptotics for  the Barnes $G$-function.

The relations between primes encoded in the factorizations of  binomial products $\G_n$ 
 seem to be of deep arithmetic significance. 
These factorizations are
described
by the  functions 
\begin{equation}\label{G-prime}
\W_{p}(\G_n):= \ord_p(\G_n),
\end{equation} 
 with $p^{\ord_p(\G_n)}$ denoting the maximal power of $p$ dividing $\G_n$.
 The prime factorizations of  the first few  $\G_n$ are
 $\G_1=1, \,\G_2=2, \,\G_3=3^2, \, \G_4=2^5 \cdot 3, \,\G_5=2^2 \cdot 5^4, 
\G_6=2^4\cdot 3^4 \cdot 5^3$
and   
$\G_7=  3^2 \cdot 5^2 \cdot 7^6$. 
 These initial values already exhibit visible oscillations in $\ord_2(\G_n)$,
 and each  function $ \ord_p(\G_n)$  separately has  a somewhat
 complicated structure of oscillations. 
Figure \ref{fig21-ord2}  plots  values of   $\ord_2(\G_n)$
 for $1 \le n \le 1023$. This plot exhibits significant structure in the  behavior of
 $\ord_2(\G_n)$,  visible as a set of stripes  in intervals between successive powers of $2$.


\begin{figure}[!htb]
\includegraphics[width=130mm]{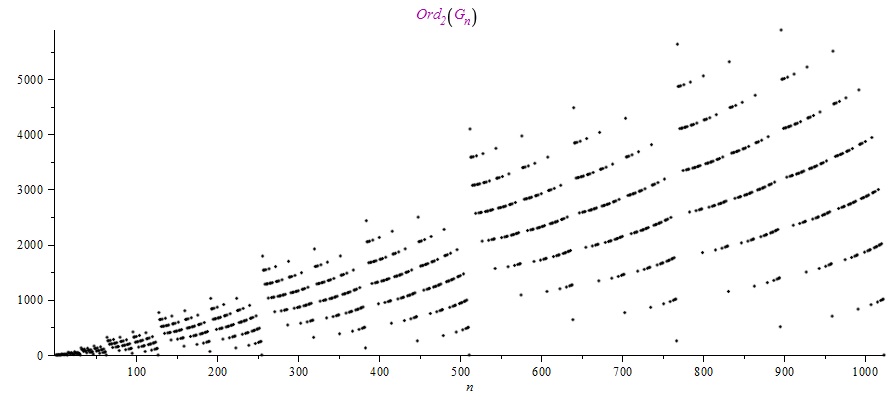}
\caption{$\W_2(n) := \ord_2(\G_n)$, $1 \le n \le 1023 =2^{10}-1$.}
\label{fig21-ord2}
\end{figure}

The behavior of the prime factorizations of $\G_n$ 
is the main focus of this paper.
We derive three  different  formulas for 
$\ord_p(\G_n)$, 
given  in Sections  \ref{sec3}, \ref{sec4} and \ref{sec5}, respectively.
Each of the formulas encodes different information about $\ord_p(\G_n)$.
The first of these formulas follows from the unreduced Farey product interpretation.
 The second of these formulas 
 relates $\ord_p(\G_n)$   to the base $p$ expansion of $n$,
which  relates to values of  the Riemann zeta
function $\zeta(s)$ on the line $Re(s) =0$  through a result of Delange \cite{Del75}. The third of these formulas
directly involves  the base $p$ radix expansion of $n$, and  
is linear and bilinear in the radix expansion digits.

The  second and third
formulas for $\W_p(\G_n)$ 
generalize  to notions attached to radix expansions to an arbitrary integer base. 
For each  $b \ge 2$ we define   integer-valued
functions $\W_b(\G_n)$ (resp. $\W_b^{\ast}(\G_n)$) for $n \ge 1$, 
which for primes $p$ satisfy
$\W_p(\G_n) = \W_p^{\ast}(\G_n) = \ord_p(\G_n)$
for all $ n \ge 1.$
In Appendix B we prove these definitions agree in general:   for all $b \ge 2$, 
\begin{equation}
\W_b(\G_n) = \W_{b}^{\ast}(\G_n) \quad \mbox{for all} \quad n \ge 1. 
\end{equation}
The functions $\W_b(\G_n)$ for composite $b$ can no longer  be interpreted as
specifying the amount of ``divisibility by $b$" of the integer $\G_n$.
It is an interesting problem to determine what arithmetic information about $\G_n$ the
functions $\W_b(\G_n)$  might encode.

From the formulas obtained for 
$\ord_p(\G_n)$ 
we  deduce results on  its  size 
 and the behavior of its fluctuations.
 We show that 
$$
0 \le \ord_p(\G_n) < n \log_p n,
$$
and that
$$\limsup_{n \to \infty}  \frac{\ord_p(\G_n)}{n \log_p n} =1.
$$
It follows  that $n \log_p n$ is the correct scale of  growth for this function. 
We also show that  each function $\ord_p(\G_n)$ oscillates infinitely many times between 
the upper and lower  bounds as $n \to \infty$.

In Section \ref{sec7} we compare the three formulas for $\ord_p(\G_n)$.
We show that between them they account for much of the 
 structure visible in the picture in Figure \ref{fig21-ord2}.

In  Section \ref{sec7b} we present  direct connections between
individual binomial products $\G_n$ and the distribution of prime numbers.
There is a tension between the smooth asymptotic growth of $\G_n$
and the oscillatory nature of the divisibility of $\G_n$ by 
individual primes. This  tension 
encodes a great deal of information about the structure of prime numbers.
Our results yield  a  Chebyshev-type  estimate for $\pi(x)$
and suggest  the possibility of a  approach to 
 the prime number theorem via radix expansion properties of $n$ to prime bases.
In another direction, a connection of the $\G_n$ to the Riemann hypothesis may   exist 
via their relation to products of Farey fractions, see  \cite{LM14r}.


%
%

\section{Unreduced Farey Fractions}\label{sec2a}

The {\em Farey sequence} $\sF_n$ of order $n$
 is the sequence of  reduced fractions $\frac{h}{k}$ between $0$ and $1$ 
(including $0$ and $1$) which, when in lowest terms, have denominators less than or equal to $n$, arranged in order of increasing size.
 It is the set 
$$
\sF_n := \{ \frac{h}{k}: 0 \le h \le k \le n: \, gcd(h,k) = 1.\}
$$
The Farey sequences encode deep arithmetic properties of the integers
and are important in Diophantine approximation, e.g. \cite[Chap. III]{HW79}.)
The distribution of the Farey fractions approaches the uniform distribution on $[0, 1]$
as $n \to \infty$ in the sense of measure theory, and the rate at which it approaches
the uniform distribution as a function of $n$ is related to the Riemann hypothesis
by a theorem of Franel  \cite{Fra24}.
Extensions of Franel's result are  given  in many later works, including Landau \cite{Lan24}, \cite{Landau27}, 
Mikol\'{a}s \cite{Mik49}, \cite{Mik51}, and Huxley \cite[Chap. 9 ]{Hux72}.

The  Farey sequences have a simpler cousin, the {\em unreduced Farey sequence} $\sG_n$, 
which is the   ordered sequence of all reduced and unreduced fractions 
between $0$ and $1$ with denominator of size at most $n$. 
We  define the positive unreduced Farey sequence by omitting the value $0$, obtaining
$$
\sG_n^{\ast}  := \{ \frac{h}{k}: 1 \le h \le k \le n\}.
$$
We let $\Phi^{\ast}(n) = | \sG_n^{\ast}|$ denote the number of positive unreduced Farey fractions,
and clearly 
\begin{equation}\label{eq-phi-ast}
\Phi^{\ast}(n) = \binom{n+1}{2} =\frac{1}{2} n (n+1).
\end{equation}
We order these unreduced fractions in increasing order, breaking ties between equal fractions 
ordering them by increasing denominator. For example, we have
$$
\sG_4^{\ast}:= \{ \frac{1}{4}, \frac{1}{3}, \frac{1}{2}, \frac{2}{4}, \frac{2}{3}, \frac{3}{4}, \frac{1}{1}, \frac{2}{2}, \frac{3}{3}, \frac{4}{4} \}.
$$
We  label the fractions in $\sG_n^{\ast}$ in this order as $\rho_r^{*} = \rho_{r,n}^{\ast}$, and write
$$
\sG_n^{\ast} = \{ \rho_{r}^{*}= \rho_{r, n}^{*}: 1\le r \le \Phi^{\ast}(n)\},
$$
Then we can  define the {\em unreduced Farey product} as
\begin{equation}\label{UFP}
G_n :=\prod_{r=1}^{\Phi^{\ast}(n)} \rho_{r,n}^{\ast}= \frac{N_n^{\ast}}{D_n^{\ast}},
\end{equation}
in which  $N_n^{\ast}$  (resp. $D_n^{\ast}$) denotes the product of the numerators 
(resp. denominators) of all $\rho_{r,n}^{\ast}$. 
The numerator function 
$$
N_n^{\ast}=\prod_{k=1}^n  k! 
$$
has been called the  {\em superfactorial function}
and appears as sequence A000178 in OEIS \cite{OEIS}.
The  denominator function 
$$
D_n^{\ast}=\prod_{k=1}^n  k^k
$$
has been called  the {\em hyperfactorial function},
and appears as sequence A002109 in OEIS \cite{OEIS}.
The hyperfactorial $D_n^{\ast}$ in expressible in terms of factorials as
\begin{equation} \label{hyper}
D_n^{\ast}  = \frac{ \prod_{k=1}^n k^n}{ 1^{n-1} 2^{n-2} \cdots (n-1)^1 \cdot n^0} 
= \frac{(n!)^n} { (n-1)! \cdots 1!} = \frac{ (n!)^n}{ N_{n-1}^{\ast}}.
\end{equation}
 It  was  studied by Glaisher \cite{Gla1878}, \cite{Gla1894},
starting in 1878.  The integers $D_n^{\ast}$  were later found  to be the sequence
of discriminants of the  Hermite polynomials\footnote{
The hyperfactorials occur as the discriminants
of the Hermite polynomials up to factor of a power of $2$ in the usual normalization 
$H_n(x)=(-1)^n e^{x^2/2} \frac{d^n}{dx^n} (e^{- x^2/2})$, see Szego \cite[ (6.71.7)]{Sz39}.
The Hermite polynomials are a fundamental family of orthogonal polynomials,
and this connection  hints at a deep importance of
the hyperfactorial function.}
 in the probabilist's normalization $He_n(x)= 2^{-n/2}H_n(\frac{x}{\sqrt{2}})$.

The unreduced Farey products  $G_n$
have  their reciprocal $\G_n = 1/G_n$ expressible in terms of binomial coefficients.

 \begin{thm}\label{th21}
The unreduced Farey product $G_n$ has its
reciprocal $\G_n =1/G_n$ given  by the  product of 
binomial coefficients
\begin{equation}\label{fullprod}
\G_n =\prod_{j=0}^n {\binom{n}{j}}.
\end{equation}
Thus $1/G_n$ is always an integer. 
\end{thm}

\begin{proof}
Enumerating the unreduced Farey fractions 
in order of fixed $k$, as 
$\frac{j}{k}$ with $1 \le  j \le k \le n$,
we have $\G_n = \frac{D_n^{\ast}}{N_n^{\ast}}$,
in which   $D_n^{\ast} = \prod_{k=1}^n  k^k$  and 
$N_n^{\ast} = \prod_{j=1}^n j! = \prod_{j=1}^n j^{n-j+1}$. Therefore,
setting $0!=1$, we have 

\begin{equation}\label{Gndef}
\G_n=\frac{1^1.2^2.3^3\dots n^n}{1^n.2^{n-1}.3^{n-2}\cdots n}=
\frac{\left(\frac{n!}{0!}.\frac{n!}{1!}\dots\frac{n!}{(n-1)!}\right)}{1!\,2!\cdots(n-1)!\,n!}\nonumber \\
=\prod_{t=1}^{n}\frac{n!}{t!(n-t)!}=\prod_{t=1}^{n}\binom{n}{t}.\\
\end{equation}
The last product also equals   $\prod_{t=0}^n \binom{n}{t}$, as required. 
\end{proof}


\begin{rem}
Products of binomial coefficients $\G_n$ appear as normalizing constants $c_{n+1}$ 
associated to  the  density $z \mapsto \frac{n+1}{\pi} (1 + |z|^2)^{-n}$ on $\mathbb{C}$, 
see Lyons \cite[Sec.3.8]{Lyons14}. This density is associated with 
a particular Gaussian orthogonal polynomial ensemble, the {\em$(n+1)$-st spherical ensemble},
which is the (randomly ordered) set of eigenvalues of $M_1^{-1} M_2$ where $M_i$ are independent $(n+1)\times (n+1)$ matrices
whose entries are independent standard complex Gaussians. 
This eigenvalue interpretation of the spherical ensemble is due  to Krishnapur \cite{Kris06}, see \cite{HKPV09}.
\end{rem}

\begin{rem}
The  reciprocal $\F_n= 1/F_n$ of the product $F_n$ of all nonzero Farey fractions  of order $n$
is a quantity analogous to $\G_n$. It encodes interesting arithmetic information, but 
is usually not an integer.  The Riemann hypothesis is encoded in its asymptotic behavior,
as discussed in Remark \ref{rem33} below.
The quantities $\F_n$ and $\G_n$ are related by the identity 
$\G_n = \prod_{k=1}^n \F_{\lfloor n/k\rfloor}$, which under a form of M\"{o}bius inversion
yields  $\F_n = \prod_{k=1}^n (\G_{\lfloor n/k \rfloor})^{\mu(k)}.$  Our study of $\G_n$ was motivated 
 in part  for its potential  to obtain useful information about $\F_n$.
\end{rem}

%
%
\section{Growth of $\G_n$}\label{sec2b}

We  estimate  the growth of $\G_n$ using 
 its connection to superfactorials $N_n^{\ast}$ and hyperfactorials $D_n^{\ast}$.
One can derive  a complete asymptotic expansion  for each of $\log(N_n^{\ast})$, $\log(D_n^{\ast})$ and 
 $\log(\G_n)$, using the Barnes G-function, which we present in Appendix A. 
 Here we derive the  first  few leading  terms,  for which Stirling's formula suffices,
 and which permit

 \begin{thm}\label{th22b}
For $n \ge 2$ the  superfactorials $N_n^{\ast}$ and hyperfactorials $D_n^{\ast}$ satisfy
\begin{eqnarray*}
\log  (D_n^{\ast}) &= & \frac{1}{2} n^2 \log n - \frac{1}{4}n^2  + \frac{1}{2} n \log n +\quad\quad\quad \quad\quad\quad\quad + \, O( \log n).\\
~&~ \nonumber\\
\log  (N_n^{\ast}) &=& \frac{1}{2} n^2 \log n - \frac{3}{4}n^2 + n \log n + \left(\frac{1}{2}\log({2 \pi})-1\right)  n +  O( \log n), 
\end{eqnarray*}
\end{thm}

\begin{proof}
We will apply  Stirling's formula in the truncated form
$$
\log (n!)=  n\log n -n + \frac{1}{2} \log n + \frac{1}{2} \log (2 \pi) + O (\frac{1}{n}),
$$
valid for all $n \ge 1$.

For the denominator term, we have 
$$
\log(D_n^{\ast}) = \sum_{k=1}^n k\log k =
 \sum_{j=1}^n \left(\sum_{k=j}^n \log k \right)= \sum_{j=1}^n \left( \log (n!) - \log (j-1)!\right).
$$
Applying Stirling's formula on the right side (and shifting $j$ by $1$) yields
$$
\sum_{k=1}^n k \log k= n \log (n!) - \sum_{j=1}^{n-1} \left(j \log j -j + \frac{1}{2} \log j + \frac{1}{2}\log (2\pi) + O(\frac{1}{j}) \right),
$$
We move the term $\sum_{j=1}^{n-1} j \log j$ to the left side and obtain
$$
2 \left(\sum_{k=1}^n k \log k \right)= n \log (n!) + n \log n + \frac{n(n-1)}{2} - \frac{1}{2} \log (n!) - \left(\frac{1}{2} \log (2\pi) \right)n + O( \log n).
$$
Applying Stirling's formula again on the right and simplifying yields
\begin{equation}\label{sumklogk}
\sum_{k=1}^n k \,\log k = \,\frac{1}{2}n^2 \log n - \frac{1}{4} n^2 + \frac{1}{2} n \log n  + O(\log n),
\end{equation}
which gives the asymptotic formula for $\log (D_n^{\ast})$ above.

For the numerator term we have
\begin{eqnarray*}
\log(N_n^{\ast}) & = & \sum_{k=1}^n \log (k!)  = \sum_{k=1}^n \Big(k \log k - k + \frac{1}{2} \log k + \frac{1}{2} \log (2 \pi) + O (\frac{1}{k})\Big) \\
 &=& \frac{1}{2} n^2 \log n -\frac{3}{4} n^2 + \frac{1}{2}n \log n  + \frac{1}{2} (\log (2 \pi)- 1) n + O (\log n).
\end{eqnarray*}
The second line used $\sum_{k=1}^n \, \log k  = \log (n!)$ with Stirling's formula,
plus the asymptotic formula \eqref{sumklogk}.
\end{proof}

We single out for special emphasis the 
initial terms in the asymptotic expansion for $\log(\G_n)$.

 \begin{thm}\label{th32b} 
The  function $\W_{\infty}(\G_n) :=\log(\G_n)$ satisfies for $n \ge 2$ the estimate
\begin{equation}\label{main-asymp1b}
\log  (\G_n)  =   \frac{1}{2}n^2-\frac{1}{2}n\log n+\left(1- \frac{1}{2} \log (2 \pi)\right) n + O(\log n).
\end{equation}
Here $1- \frac{1}{2} \log (2 \pi) \approx 0.08106$.
\end{thm}

\begin{proof}
The  asymptotic formula for $\log (\G_n) = \log (D_n^{\ast}) - \log(N_n^{\ast})$  follows immediately
from Theorem \ref{th22b}. 
\end{proof}

\begin{rem}\label{rem32}
This expansion  captures
a connection to density of primes and 
has a further analogy with the Riemann hypothesis, given the next remark.
It shows that the  function $\log(\G_n)$  is asymptotic   
to  $\frac{1}{2} n^2$, which is  smaller 
by a  logarithmic factor than  either of $\log(D_n^{\ast})$ or $\log(N_n^{\ast})$ separately.
That is, the top terms in the asymptotic expansions of $\log(D_n^{\ast})$ or $\log(N_n^{\ast})$ cancel.
 This savings of a  logarithmic factor in the main term of the asymptotic formula is directly
related to primes having density $O( \frac{n}{\log n})$, and to  obtaining  Chebyshev-type
bounds for $\pi(x)$, see Section \ref{sec7b}.
\end{rem}

\begin{rem}\label{rem33}

The  analogy of  the 
asymptotic formula \eqref{main-asymp1b}
 with the Riemann hypothesis 
arises from  its interpretation in terms of products of unreduced Farey fractions
and concerns  its remainder term $O(\log n)$.
We can rewrite it  in terms
of the number $\Phi^{\ast}(n)= \binom{n+1}{2}$ of unreduced Farey fractions as 
\begin{equation}\label{main-asymp2}
\log (\G_n)=\Phi^{\ast}(n)-\frac{1}{2} n\log n + \big(\frac{1}{2} -\frac{1}{2} \log (2 \pi)) \big) n + O ( \log n),
\end{equation}
with $\frac{1}{2} - \frac{1}{2} \log (2 \pi) \approx -041894$.
This expression is directly comparable with an expression  for the 
 logarithm of (inverse)  Farey products $\log(\F_n)$  having the form
\begin{equation}\label{Miko}
\log (\F_n) = \Phi(n) - \frac{1}{2} n + R(n),
\end{equation}
in which  $\Phi(n)$ counts
 the number of Farey fractions and $R(n)$  is a  remainder term
 defined by the equality \eqref{Miko}.
The function  $\Phi(n) = \sum_{k=1}^n \phi(k)$ is  the summatory function
for the Euler $\phi$-function, and satisfies $\Phi(n) \sim \frac{3}{\pi^2} n^2$
as $n \to \infty$.
 In 1951  Mikol\'{a}s \cite[Theorem 1]{Mik51},
showed   that the Riemann hypothesis is equivalent to the assertion that 
the remainder term is small, satisfying 
$$ 
R(n) = O( n^{\frac{1}{2} + \epsilon})
$$
 for each $\epsilon >0$ for $n \ge 2$. In fact he showed that
 estimates of form $R(n) =O( x^{\theta+ \epsilon})$ for fixed $1/2 \le \theta <1$
and for all $\epsilon >0$  were
equivalent to a zero-free region of the Riemann zeta function for
$Re(s)> \theta$.
We can therefore view  \eqref{main-asymp2} 
by analogy as an ``unreduced Farey fraction Riemann hypothesis" 
in view of its small error term
 $O(\log n)$. 
A Riemann hypothesis type estimate would
 require only an error term of form $O( n^{1/2 + \epsilon})$.

See \cite[Section 3]{LM14r} for a further discussion of Mikolas's results,
which include unconditional error bounds for $R(n)$. The true subtlety in the Mikol\"{a}s
formula seems to resolve around oscillations in the function $\Phi(x)$  
of magnitude at least $\Omega (x \sqrt{\log\log x})$  which themselves are
related to zeta zeros.
\end{rem}

%
%
%
\section{Prime-power divisibility of $\G_n$: Formulas using integer parts}\label{sec3}

 We study the divisibility  of $\G_n$ by powers of a fixed prime.
We obtain three distinct formulas for $\ord_p(\G_n)$, in this section  and 
in the following two sections, respectively. 

The first formula simply encodes the Farey product decomposition.

\begin{thm}\label{th31a} 
 For $p$ a prime, the function $\W_p(G_n) :=\ord_p (\G_n)$ satisfies
\begin{equation}\label{ordp-big1}
\W_p(\G_n) = \ord_p(D_n^{\ast}) - \ord_p(N_n^{\ast}),
\end{equation}
where $D_n^{\ast} = \prod_{k=1}^n k^k$ and $N_{n}^{\ast} = \prod_{k=1}^{n} k!.$ 
\end{thm}

\begin{proof}
This formula follows directly 
by applying $\ord_p(\cdot)$
to both sides of the decomposition $\G_n= \frac{1}{G_n} = \frac{D_n^{\ast}}{N_n^{\ast}}.$
\end{proof}

The formula \eqref{ordp-big1} has several interesting features.
\begin{enumerate}
 \item[(i)]
This formula expresses $\ord_p(\G_n)$ 
as a difference of  two positive terms,
$$
\SP_{p,1}^{+}(n)  := \ord_p( D_n^{\ast})= \sum_{k=1}^n    \ord_p (k^k)
$$ 
and
$$
\SP_{p,1}^{-}(n) :=  \ord_p( N_{n}^{\ast})= \sum_{k=1}^n \ord_p (k!).
$$
Both terms are nondecreasing in $n$, that is,   
$$\Delta ( \SP_{p,1}^{\pm})(n):= \SP_{p,1}^{\pm}(n) - \SP_{p,1}^{\pm}(n-1)$$
 are  nonnegative functions.
Furthermore  the difference term $\Delta (\SP_{p,1}^{-})(n)$ is nondecreasing in $n$.

\item[(ii)] There is  a race in size between the terms $\SP_{p,1}^{+}(n)$ and $\SP_{p,1}^{-}(n)$,
as $n$ varies. 
The first term $\SP_{p,1}^{+}(n)$ jumps only when $p$ divides $n$ and makes 
large jumps at these values. In contrast, the  second term $\SP_{p,1}^{-}(n)$ changes  in smaller nonzero 
increments, making a positive contribution whenever $p \nmid n$ and $n > p$.  In consequence: For $n \ge p$, { $\ord_p (\G_n)$ 
 increases going from $n-1$ to $n$
when $p \mid n$, and  strictly decreases when $p \nmid n$.}
\end{enumerate}

We may re-express the terms in formula \eqref{ordp-big1} using the floor function (greatest integer part function).
 We  start with {\em de Polignac's formula} (attributed to Legendre by Dickson \cite[p. 263]{Dick19}), which states that 
$$
\ord_{p}(n!) = \sum_{j=1}^{\infty} \Big\lfloor \frac{n}{p^j} \Big\rfloor\, .
$$
The sum on the right  is always finite, with largest nonzero term $j=\lfloor \log_p n \rfloor$, with $p^j \le n < p^{j+1}$.
We obtain
$$
\ord_p(N_n^{\ast} )= \sum_{k=1}^n \Big(\sum_{j=1}^{\infty} \Big\lfloor \frac{k}{p^j} \Big\rfloor\Big)
$$
and, using \eqref{hyper},
$$
\ord_p(D_n^{\ast}) = n \left( \sum_{j=1}^{\infty} \Big\lfloor \frac{n}{p^j} \Big\rfloor \right) - \ord_p( N_{n-1}^{\ast}).
$$

We next obtain asymptotic estimates with error term for  $\ord_p(N_n^{\ast})$ and $\ord_p(D_n^{\ast})$,
and use these estimates to upper bound the size of $\ord_p(\G_n)$.

\begin{thm}\label{th32a} 
 For $p$ a prime, and all $n \ge 2$, 
$$
\ord_p(N_n^{\ast})= \frac{1}{2(p-1)} n^2 +O\Big(n \log_p n\Big).
$$
and
$$
\ord_p(D_n^{\ast})= \frac{1}{2(p-1)} n^2 +O\Big(n \log_p n \Big).
$$
It follows that, 
$$
\ord_p(\G_n)= O \Big( n \log_p n \Big).
$$
for all $n \ge 1$. 
\end{thm}

\begin{proof}
We rewrite de Polignac's formula using the identity
 $\frac{n}{p^j} = \lfloor \frac{n}{p^j} \rfloor + \{ \frac{n}{p^j} \}$, with the
fractional part function $\{ x\} := x - \lfloor x\rfloor$,  to obtain
\begin{equation} \label{mod-de-pol}
\ord_p(n!) = \frac{n}{p-1} - \sum_{j=1}^{\infty} \{ \frac{n}{p^j} \}.
\end{equation}
For $j > \log_p n$ we have $\{ \frac{n}{p^j} \}= \frac{n}{p^j}$ 
so the series becomes a geometric series  past this point
and can be summed.  One obtains the estimate
$$
\ord_p (n!) = \frac{n}{p-1} + O ( 1+\log_p n),
$$
with an $O$-constant independent of $p$.

For $N_n^{\ast}$ we have
$\ord_p(N_n^{\ast})= \sum_{k=1}^N \ord_p (k!),$
and applying \eqref{mod-de-pol} yields
\begin{eqnarray*}
\ord_p(N_n^{\ast}) &= &\sum_{k=1}^n \frac{k}{p-1} - \sum_{k=1}^n  \Big(\sum_{j=1}^{\infty} \{ \frac{k}{p^j} \}\Big) \\
&=& \frac{n^2+n}{2(p-1)}  + O \Big ( n(1+\log_p n) \Big)
\end{eqnarray*}
The result follows by shifting $\frac{1}{2(p-1)}n$ to the remainder term.

For $D_n^{\ast}$ we have, using \eqref{hyper}, that 
\begin{equation}\begin{aligned}
\ord_p(D_n^{\ast}) &=  n \ord_p( n!) - \ord_p(N_{n-1}^{\ast})\\
&= \left( n \left(\frac{n}{p-1}\right) + O( 1+\log_p n) \right)\\
&- \left(\frac{1}{2(p-1)} (n-1)^2 +O\Big(n(1+ \log_p n)\right)\\
&= \frac{1}{2(p-1)} n^2 +O\Big(n\log_p n\Big).
\end{aligned}\end{equation}

For $\G_n^{\ast}$ the result follows from $\ord_p(\G_n) = \ord_p(D_n^{\ast}) - \ord_p(N_n^{\ast}).$
\end{proof}

The bound $\ord_p(\G_n)= O (n \log_p n)$, valid for all $n \ge p$,
  quantifies the smaller size of
   in  size of $\log (\G_n)$ compared to either $\log(N_n^{\ast})$ and $\log(D_n^{\ast})`$.
 In the situation here  the smaller size is by almost a square root factor.
The bound on $\ord_p(\G_n)$  above is the correct order of magnitude,
and we obtain a sharp constant in Theorem \ref{th29} below. 


%
%
%

\section{Prime-power divisibility of  $\G_n$: Formulas using base $p$ digit sums}\label{sec4}

We obtain  a second formula for $\ord_p(\G_n)$, one expressed  directly in terms of base $p$ digit sums,  
and draw consequences. We start from
\begin{equation} \label{eq301}
\ord_p(\G_n) = \sum_{k=0}^n \ord_p \binom{n}{k}.
\end{equation}
The divisibility of binomial coefficients by prime powers $p^k$ has been
studied for over $150$ years, see  the extensive survey of  Granville \cite{Gra97}.
Divisibility properties  are well known to be related to  the coefficients $a_j$ of the the base $p$ radix expansion of $n$,
written as
$$
n = \sum_{j=0}^k  a_j p^j, \quad  0 \le a_j \le p-1,
$$
with $k = \lfloor \log_p n\rfloor.$

%
%
%

\subsection{Prime-power divisibility of $\G_n$: digit summation form}\label{sec33}

We derive a formula for $\ord_p(\G_n)$ that expresses it 
in terms of summatory functions of base $p$ digit sums.  

We will consider  digit sums  more generally  
 for radix expansions to an  arbitrary integer base $b \ge 2$.
Write a positive integer $n$ in  base $b \ge 2$ as
$$
n := \sum_{i=0}^{k} a_i b^i, \, \mbox{for} \,\, \,  b^k \le n < b^{k+1}.
$$
with digits $0 \le a_i \le b-1$. Here $k = \lfloor \log_b n\rfloor.$
and $a_i := a_i(n)$ with $a_k(n)\geq1$.

\begin{enumerate}
\item
The {\em sum of digits function $\dgt_b(n)$ (to base $b$)} of $n$  is
\begin{equation}\label{sum-dig}
\dgt_b(n) := \sum_{i=0}^k a_i(n),
\end{equation}
with  $k = \lfloor \log_b n\rfloor.$\smallskip

\item
 The {\em running digit sum  function $\Sum_b(n)$  (to base $b$)}  is
\begin{equation}\label{tot-sum-dig}
\Sum_b(n) := \sum_{j=0}^{n-1} \dgt_b(j).
\end{equation}
\end{enumerate}

Our  second formula
for $\ord_p(\G_n)$ is given in terms of these quantities;
we defer its proof to the end of this subsection.


\begin{thm}\label{th39} 
Let the prime $p$ be fixed. Then for all $n \ge 1$, 
\begin{equation}\label{summatory}
\W_p(\G_n) :=\ord_p (\G_n) = \frac{1}{p-1} \Big(2\Sum_p(n) - (n-1) \dgt_p(n)   \Big).
\end{equation}
\end{thm}

  The  formula \eqref{summatory} has several  interesting features.
  \begin{enumerate} 
\item[(i)]
This formula  expresses $\ord_p(\G_n)$
 as a difference   of two positive terms,
\begin{equation}\label{sp2-plus}
\SP_{p, 2}^{+}(n) := \frac{2}{p-1} \Sum_p(n)
\end{equation}
and 
\begin{equation}\label{sp2-minus}
\SP_{p, 2}^{-}(n) := \frac{n-1}{p-1} \dgt_p(n).
\end{equation}
 The two functions,
 $\Sum_p(n)$ and $\dgt_p(n)$ have been extensively
 studied in the literature. They exhibit
 very different behaviors as $n$ varies:  $\Sum_p(n)$  grows rather smoothly  while 
  $\dgt_p(n)$ exhibits large  abrupt variations in size.
 \item[(ii)] 
 The function  $\Sum_p(n)$ has smooth variation and obeys the
asymptotic estimate   
$$\Sum_p(n) = (\frac{p-1}{2})n \log_p n + O(n),$$
see Theorem \ref{th313}.
In consequence $\SP_{p,2}^{+}(n) = n \log_p n +O(n)$. Furthermore $\Sum_p(n)$
can itself be written as a difference of two positive functions using the identity
\eqref{279} below, and noting the second term is nonpositive, by  Theorem \ref{th48a} (1).
\item[(iii)]
 The function  $\dgt_p(n)$ is known to have  average size 
 $\frac{p-1}{2}  \log_p (n)$ but is oscillatory.  For most $n$  it is  rather close to its average size, 
however it varies  from $1$ to a value 
as large as $ (p-1)\log_p n$ infinitely often as $n \to \infty$,
as given by the distribution of $\dgt_p(n)$ as $n$ varies. If one takes $n=p^k$ and
samples $m$ uniformly on  the range $[1, p^k]$, then $\dgt_p(m)$ 
it is a sum of  $k$ identically distributed independent random variables,
and as $k \to \infty$ will obey a central limit theorem. One can show that
it has size sharply concentrated around $ (\frac{p-1}{2})\log_p n$ with a spread
on the order of $C_p \sqrt{k}(\frac{p-1}{2})$.
In consequence, the second term $\SP_{p, 2}^{-}(n) = \frac{n+1}{p-1}\dgt_p(n)$
is positive and  has  average size  $\frac{1}{2}  n \log_p n + O(n)$,
which is in magnitude  half that of the first term. It has large variations in size,
between being twice its average size and being $o(n \log n)$.
\item[(iv)]
The function $\dgt_p(n)$ is highly correlated between successive values of $n$.
It  exhibits an "odometer" behavior where it has increases
by one at most steps, but has jumps downward of size 
about $p^k$ at values of $n$ that $p^k$ exactly divides. 
\end{enumerate}

To derive Theorem \ref{th39}, we  make use of the following elegant formula for 
 $\ord_p \binom{n}{t} $ noted by Granville \cite[25, equation following (18) ]{Gra97}.

\begin{prop}\label{prop310} 
For $n \ge 1$ and $0 \le t \le n$, 
\begin{equation}\label{ord-binom}
\ord_p \binom{n}{t}  =\frac{1}{p-1}\big( \dgt_p(t)+\dgt_p(n-t)-\dgt_p(n)\big).
\end{equation}
\end{prop}

\begin{proof}
Writing $n = \sum_{i=0}^k a_i p^i$, and applying de Polignac's formula, we have
\begin{eqnarray*}
\ord_p(n) &= &\frac{n- a_0}{p} + \frac{n- (a_1 p + a_0)}{p^2} + \cdots + \frac{n -(a_{k-1}p^{k-1} + \cdots + a_0)}{p^k}\\
&& ~~~~~+ \sum_{i=k+1}^{\infty} \frac{n -(a_k p^k + a_{k-1}p^{k-1} + \cdots + a_0)}{p^i},
\end{eqnarray*}
in which all  the terms in the last sum are identically zero. Collecting the terms for $n$ and for each $a_i$ separately
on the right side of this expression, 
each forms a geometric progression, yielding
\begin{equation}\label{277}
\ord_p(n!) =  \frac{1}{p-1}\big( n - (a_k+ a_{k-1} + \cdots + a_0) \big)= \frac{1}{p-1}\big( n -\dgt_p(n)\big).
\end{equation}
Writing the binomial coefficient $\binom{n}{t}= \frac{n!}{t! (n-t)!} $
 and substituting   \eqref{277} above
yields the desired formula. 
\end{proof}

\begin{proof}[Proof of Theorem \ref{th39}.]
Combining Theorem \ref{th21} with Proposition \ref{prop310} and noting that $\dgt_p(0)=0$, we have
\begin{eqnarray*}
\ord_p (\G_n) &= & \sum_{t=0}^n \ord_p \binom{n}{t} = \frac{1}{p-1} \sum_{t=0}^n ( \dgt_p(t)+\dgt_p(n-t)-\dgt_p(n))\\
&=& \frac{1}{p-1} \big(2 \sum_{t=0}^n \dgt_p(t)   - (n+1) \dgt_p(n) \big)\\
&=& \frac{1}{p-1} ( 2 S_p(n) - (n-1) \dgt_p(n)),
\end{eqnarray*} 
as required.
\end{proof}

%
%
%

\subsection{Analogue  function $\W_b(\G_n)$ for a general radix base $b$ }\label{sec52}

The  functions of digit sums on 
the right side of \eqref{summatory}  make sense for all radix bases $b \ge 2$,
which leads us to define general functions $\W_b(\G_n)$ for $b \ge 2$.

\begin{defn}\label{def53}
For each integer $b \ge 2$ and  $n \ge 1$,  the {\em generalized order}
$\W_b(\G_n)$ of $\G_n$ to base $b$ is
\begin{equation}\label{new-summatory}
\W_b(\G_n) := \frac{1}{b-1} \Big(2\Sum_b(n) - (n-1) \dgt_b(n)   \Big).
\end{equation}
\end{defn}
Theorem \ref{th39} shows that  for prime $p$ we have  $\W_p(\G_n)= \ord_p(\G_n)$. 
However for  composite $b$ the function $\W_b(\G_n)$ 
does not always   coincide with the largest power of $b$ dividing $\G_n$,
even for  $b=p^k \, (k \ge 2)$ a prime power, i.e. for composite $b$ 
 $\W_b(\G_n) \ne \ord_b(\G_n)$ occurs for some $n$.

One may obtain an upper bound for $\W_b(n)$ using an upper bound
for the  running digit sum  function $\Sum_b(n)= \sum_{m=0}^{n-1} d_b(m)$.
In 1952 Drazin and Griffith \cite{DG52} obtained the following
sharp upper bound, as a special case of more general results.


\begin{thm}\label{th54a} {\em (Drazin and Griffith (1952))}
Let $b \ge 2$ be an integer. Then for all $n \ge 1$, 
\begin{equation}\label{sum-bound}
\Sum_b(n) \le \frac{b-1}{2} n \log_b n.
\end{equation}
and equality holds  if and only if  $n= b^k$ for  $k \ge 0$.
\end{thm}

\begin{proof}
This result is \cite[Theorem 1]{DG52}, taking $t=1$, 
asserting $\Delta_1(b, n) \ge 0$ with equality for $n= b^k$. In their notation 
$\sigma_1(b) = (b-1)/2$, 
$F_1(b, n) = \frac{b-1}{2} n \log_b n$ and 
$\Delta_1 (b, n) = \frac{2}{b-1}( F_1(b, n) - \Sum_b(n)).$

\end{proof}

We deduce the following upper bound for the generalized order to base $b$.


\begin{thm}\label{th54b} 
Let $b \ge 2$ be an integer.  Then for all $n \ge 1$, 
\begin{equation}\label{sum-bound}
\W_b(\G_n) \le  n \log_b n - \frac{n-1}{b-1}.
\end{equation}
\end{thm}

\begin{proof}
Using the definition and Theorem \ref{th54a} we  have
\begin{eqnarray*}
\W_b(\G_n)  & =  &\frac{1}{b-1} \Big( 2\Sum_b(n) - (n-1) \dgt_b(n)   \Big) \\
&\le & n \log_b n - \frac{n-1}{b-1}\dgt_b(n)  \le n \log_b(n) - \frac{n-1}{b-1},
\end{eqnarray*}
as asserted. 
\end{proof}

For the case that $b=p$ is prime, 
we obtain a slight improvement on this upper bound in
Theorem \ref{ordpgapk-1}.

%
%
%

\subsection{Summatory function of base $b$ digit sums: Delange's theorem}\label{sec34}

 The detailed behavior of the running digit sum function $\Sum_b(n)=\sum_{j=0}^{n-1} d_b(n)$ 
 has complicated, interesting properties.
In 1968 Trollope \cite{Tro68} obtained an exactly describable closed form for $\Sum_2(n)$.
In 1975 Delange \cite{Del75} obtained the following definitive result 
applying to   $\Sum_b(n)$ for  all bases $b \ge 2$.


\begin{thm}\label{th313} {\em (Delange (1975))}
Let $b \ge 2$ be an integer. 

(1) For all integers $n \ge 1$, 
\begin{equation} \label{279}
 \Sum_b(n) = \big(\frac{b-1}{ 2}\big) n \log_b n + \f_b( \log_b n) n,
\end{equation}
in which $\f_b(x)$ is a  continuous real-valued function
which is  periodic of period $1$.

(2) The function $\f_b(x)$ has a Fourier series expansion
$$
\f_b(x) = \sum_{k \in \ZZ}  c_b(k) e^{2 \pi i k x},
$$
 whose Fourier coefficients
are, for $k \ne 0$,
\begin{equation}\label{zetac}
c_b(k) = -\frac{ b-1}{2 k \pi i} \big( 1+ \frac{2 k \pi i}{\log b} \big)^{-1} \zeta( \frac{2 k \pi i}{\log b})
\end{equation}
and, for the  constant term $k=0$, 
\begin{equation}\label{CT}
c_b(0) = \frac{b-1}{2 \log b}( \log(2 \pi)-1) - (\frac{b+1}{4}).
\end{equation}
The function $\f_b(x)$ is continuous but not differentiable.
\end{thm}

\begin{proof}
(1) This statement is the main Theorem\footnote{
Delange uses a different notation for digit sums.  He writes $S_p(n)$ for
the function that we call $\dgt_p(n)$.}  of 
Delange \cite[p. 32]{Del75}.

(2) The explicit Fourier series expression is
derived  in Section 4 of \cite{Del75}. The Fourier coefficients are complex-valued
with  $c_{-k} = \bar{c}_k$, as is required for a real-valued function $\f_b(x)$.
The Fourier series coefficients in  \eqref{zetac} for $\f_{b}(x)$ 
 involve values of the Riemann zeta function evaluated at points on the line $Re(s) =0$ 
 which are, $b=p$ a prime  the poles of the Euler product factor 
at $p$ in the Euler product for $\zeta(s)$, which is $(1- \frac{1}{p^s})^{-1}$.
The growth rate of the Riemann zeta function on the line $Re(s)=0$ states that 
$|\zeta (i t)|= O ( (1+ |t|)^{1/2 + \epsilon})$, which bounds the Fourier coefficients
sufficiently to prove that $\f_b(x)$ is a continuous function. 
Delange deduces
the everywhere non-differentiable
property of $\f_b(x)$ from
a self-similar functional relation that $\f_b(x)$ satisfies.
\end{proof}

\begin{rem}
Delange proved Theorem \ref{th313} using methods from real analysis. Different
approaches were introduced 
 in 1983  by Mauclaire and Murata (\cite{MM83a},  \cite{MM83b}),
and in 1994 by Flajolet et al. \cite[Theorem 3.1]{FGKPT94},   using 
complex analysis and Mellin transform techniques.
The latter methods obtain the Fourier expansion of $\f_b(x)$
but do not establish the non-differentiability properties
of the function $\f_b(x)$.  In another direction, in 1997 G. Tenebaum \cite{Ten97} extended
the non-differentiability property of $\f_b(x)$ to other periodic functions  arising 
 from summation formulas in a similar fashion. 
Generalizations  of the Delange function connected to  higher moments of
digit sums were studied by Coquet \cite{Coq86} and Grabner and Hwang \cite{GH05}.
 \end{rem}
 
We next show that the  function $\f_b(n)$ is  nonpositive,
and give some estimates for its size, 
using results of Drazin and Griffiths \cite{DG52}.

\begin{thm}\label{th48a} 

(1) For  integer $b \ge 2$ and all real $x$, 
\begin{equation}\label{481}
\f_b( x) \le 0,
\end{equation}
and equality $\f_b(x) =0$ holds only for $x=n$, $n \in \ZZ$.

(2) For integer $b \ge 3$ and all real $x$, 
\begin{equation}\label{482}
\frac{2}{b-1}|\f_b(x)| \le \frac{b-1}{b-2} \frac{ \log (b-1)}{\log b}.
\end{equation}
\end{thm}

\begin{proof}

(1) Drazin and Griffiths \cite{DG52} study a function
$\Delta_1(b, n)$ for integer $(b, n)$ which is exactly
$\Delta_1(b, n)= -\frac{2}{b-1} \f_b(\log_b n)$ 
The nonpositivity for $x \in [0,1]$   follows from \cite[Theorem 1]{DG52}, 
using the fact that $F_b(x)$ is periodic of period $1$
and a continuous function, together with the fact that 
the  fractional
parts of $\log_b n$ are dense in $[0,1]$.

(2) This result follows from the bound on $\Delta_1(b, n)$ 
in \cite[Theorem 2]{DG52} in a similar fashion.
\end{proof} 


\begin{figure}[!htb]
\includegraphics[width=120mm]{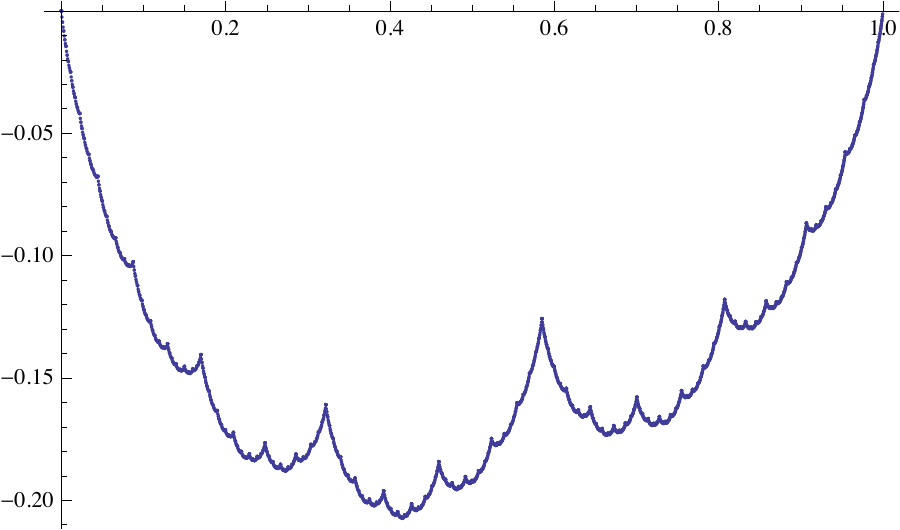}
\caption{The periodic function $\f_2(x)$.}
\label{fig31-delange}
\end{figure}

 Figure \ref{fig31-delange}  presents  a picture of the function $\f_2(x)$  computed
by J. Arias de Reyna. 
It shows the non-positivity of $\f_2(x)$ but 
 can only hint at the property of $\f_2(x)$  being non-differentiable
at every point. 
In fact  $\f_2(x)$ is related to a famous everywhere non-differentiable function,
the Takagi function $\tau(x)$, introduced by Takagi \cite{Tak03} in 1903,
which is given by
\begin{equation}\label{takagi}
 \tau(x) := \sum_{n=0}^{\infty} \frac{1}{2^n}\ll 2^n x \gg,
\end{equation}
where $\ll x \gg$ is the distance from $x$ to the nearest integer.
 This connection can be deduced from work of Trollope \cite{Tro68},
as explained in   \cite[Theorems 9.1 and 9.2]{Lag12}. One finds that
\begin{equation}
\f_2(x) = -\frac{1}{2} \ \Big( \frac{ \tau(2^x -1)}{2^x} + \frac{2^x -1}{2} - \frac{2^x -1}{2^x}\Big),  \quad 0 \le x \le 1.
\end{equation} 
Although $\tau(x)$ is everywhere non differentiable, its oscillations 
on small scales are known to not be too large.
There is a constant $C$ such that for all real $x$, 
$$
|\tau(x+h) - \tau(x)| \le C |h| \log \frac{1}{|h|}, \quad  \mbox{for all}\quad  |h| \le \frac{1}{2}.
$$


For our application to $\G_n$, $b=p$ is a prime,   the formula  \eqref{279} for $S_b(n)$ gives a smooth
``main term"  
 $\frac{p-1}{2} n \log_p n$ and  a slowly oscillating ``remainder term"   $R_p(n) := \f_b (\log_b n) n$
of order $O(n)$,   with an explicit constant in the $O$-symbol, which encodes a logarithmic
rescaling of the value of $n$.

%
%
%

\section{Prime-power divisibility of  binomial products $\G_n$: Formulas using fractional parts}\label{sec5}

%
%
%

\subsection{Prime-power divisibility of  $\G_n$: bilinear radix expansion  form}\label{sec31}

We give a third formula for $\ord_p(\G_n)$, which is also based on   the base $p$ radix expansion of $n$,
but which expresses  it as  a linear and bilinear expression in its base  $p$ coefficients.  
\begin{thm}\label{opgn} 
 Let $p$ be prime and write the base $p$ expansion of  $n = \sum_{j=0}^{k} a_j p^j$, with $a_k \ne 0$. Then
\begin{equation}\label{ordgndef2}
\ord_p(\G_n)= \sum_{j=1}^kja_jp^j-\left(\sum_{j=1}^ka_j\left(\frac{p^j-1}{p-1}\right)+
\sum_{j=0}^k\frac{1}{p^{j+1}}\Big(\sum_{\uu=0}^j a_{\uu}p^\uu\Big)\Big( \sum_{v=j+1}^ka_vp^{v}\Big)\right).
\end{equation}
\end{thm}

\indent The formula \eqref{ordgndef2} has several interesting features.
\begin{enumerate}
\item[(i)]
The formula expresses $\ord_p(\G_n)$ as  a difference 
of two positive terms, 
\begin{eqnarray*}
 \SP_{p,3}^{+}(n) &:= & \sum_{j=1}^kja_jp^j\\
 \mbox{and} \quad\quad\quad& ~\\
\SP_{p,3}^{-}(n) & := & \sum_{j=1}^ka_j\frac{p^j-1}{p-1}+\sum_{j=0}^k\frac{1}{p^j}\Big(\sum_{\uu=0}^ja_{\uu}p^\uu\Big)\Big( \sum_{v=j+1}^ka_vp^{v}\Big).
\end{eqnarray*}
An immediate consequence is the upper bound 
$\ord_p(\G_n) \le S_{p,3}^{+}(n)$
which is useful in  obtaining   upper bounds for $\ord_p(\G_n)$, see Section \ref{sec32}.
\item[(ii)]
The first two sums on the right   in \eqref{ordgndef2} are {\em linear} functions of the base $p$ digits of $n$, while  the third sum is {\em bilinear} in
the base $p$ digits.
\item[(iii)]
The first sum on the right  in \eqref{ordgndef2} makes sense as a $p$-adic function.
That is, it extends continuously to  a $p$-adically convergent series for $n \in \ZZ_p$, the $p$-adic integers.  
However the last two sums on the right, treated separately or together, do not have continuous extensions to $\ZZ_p$.
\end{enumerate}

The right side  of \eqref{ordgndef2} makes sense for arbitrary bases $b \ge 2$, so
we  make the following definition for arbitrary $b$.

\begin{defn} \label{def62}
For each integer $b \ge 2$ and  $n \ge 1$ set 
\begin{equation*}
\W_b^{\ast}(\G_n) := \sum_{j=1}^kja_jb^j-\left(\sum_{j=1}^ka_j\left(\frac{b^j-1}{b-1}\right)+
\sum_{j=0}^k\frac{1}{b^{j+1}}\Big(\sum_{\uu=0}^ja_{\uu}b^{\uu}\Big)\Big( \sum_{v=j+1}^ka_vb^{v}\Big)\right),
\end{equation*}
in which  $n=\sum_{j=0}^k a_j b^j$ is its base $b$ radix expansion.
\end{defn}

For $b=p$ a prime,  we  have $\W_p^{\ast}(\G_n) = \ord_p(\G_n)$ by Theorem \ref{opgn}.
A priori this definition looks different from that of the generalized order $\W_b(n)$ to
base $b$ introduced in Section \ref{sec4}, but in 
 Appendix B we show they coincide:  {\em For each $b \ge 2$ one has
 $$
  \W_b^{\ast}(\G_n)=\W_b(\G_n) \quad \mbox{ for all}\quad  n \ge 1.
  $$
 }

We will deduce Theorem \ref{opgn}
starting from Kummer's formula
for the maximal power of $p$ dividing a binomial coefficient  (Kummer \cite{Kum1852}).

\begin{thm}\label{th25}  {\em (Kummer (1852))}
Given a prime $p$, the exact divisibility $p^e$ of $\binom{n}{t}$ by a power of $p$  is found by writing
$t$, $n-t$ and $n$ in base $p$ arithmetic. Then  $e$ is the number of carries
that occur when adding $n-t$ to $t$ in base $p$ arithmetic, using digits
$\{ 0, 1, 2, \ldots, p-1\}$, working from the least significant digit upward.
\end{thm}

\begin{proof} 
 Kummer's theorem easily follows   from Proposition \ref{prop310} as Granville \cite{Gra97} observes.
By inspection we  see that each carry operation adding $t$ to $n-t$ in the $j$-th place reduces the sum of the digits 
in the sum $n$ by $p-1$, since it adds a $1$ in the $(j+1)$-st place while removing a sum of $p$ in the $j$-th
place. Thus the formula on the right in \eqref{ord-binom} counts the number of carries made in adding $t$ to
$n-t$.
\end{proof}

To establish Theorem \ref{opgn} we  first  reinterpret Kummer's formula  as counting the number of borrowings involved in subtracting
$j$ from $n$ in base $p$ arithmetic, now  working from the least significant digit upwards.
As an example, take base $p=3$ and consider
$$
n=  13= (111)_3, \quad t=5 = (12)_3 \quad \mbox{and} \quad  n-t= 8 =(22)_3.
$$
In
the following table we add $(n-t)$ to $t$ on the left  and subtract $t$ from $n$  on the right.
We list  the carries  ($+1$) and the borrowings  ($-1$) on the top line of the table.
 
\begin{equation*}
\begin{array}{|c|ccc| c |ccc|}
\hline
\mbox{carries}: &{\it 1} & {\it 1} & {\it 0} & \mbox{borrows}: & {\it -1} & {\it -1} & {\it 0} \\
\hline
\hline
~&~ & 2 & 2 & \quad \quad  & 1 & 1 & 1 \\
~&+ & 1 & 2 &  \quad \quad & - &  1 & 2 \\
\hline
~&1 & 1 &1 & \quad \quad & 0 & 2 & 2 \\
\hline
\end{array} 
\end{equation*}
 In this example there  are $2$ carries in the additive form versus  $2$ borrowings in the subtractive form.
 
 We will derive the  formula \eqref{ordgndef2}  of Theorem \ref{opgn} by
keeping track of the  total number of borrowings in the addition made for the $j$-digit of $n$,
but treating the contributions to each digit separately, specified  by a function $c_j(n)$ defined below.
For a given $n$ and $0 \le t \le n$ set the {\em $j$-th carry digit} $c_j(n, t)= 1$ or $0$
according as whether
the addition of  $t$ to $n-t$ in base $p$ expansion has a carry digit, or not, 
added to the $(j+1)$-place from the $j$-th place.  Equivalently $c_j(n, t)$ specifies whether there is  a borrowing
from the $(j+1)$-st place  in subtracting $t$ from $n$ in  base $p$ arithmetic.
Here $c_j(n, t)$ depends only on   $n \, (\bmod \, p^{j+1})$ and $t \, (\bmod\, p^{j+1})$.
The table above computes that $c_2(13, 5) =0$.


\begin{defn}\label{de33}
 Let a prime $p$ be fixed, and let a digit position  $j \ge 0$
be given.  The {\em $j$-th position total carry function} $c_j(n)$ is 
 \begin{equation}\label{641}
 c_j(n) : = \sum_{t=0}^n c_j(n, t).
 \end{equation}
\end{defn}
 
Kummer's theorem applied to \eqref{eq301} yields 
\begin{equation}\label{carries}
\ord_p(\G_n) = \sum_{j=0}^{\infty} c_j(n).
\end{equation}
The sum on the right is always finite since  $c_j(n) =0$ for all $ j \ge k = \lfloor \log_p n\rfloor$.


\begin{lem}\label{cjdef}
(1) For a fixed prime $p$,
\begin{equation}\label{cjdef1}
c_j(n)=\left((p^{j+1}-1)-\sum_{\uu=0}^ja_{\uu}p^{\uu}\right) \left(\sum^k_{t=j+1}a_tp^{t-j-1}\right).
\end{equation}

(2) Alternatively we  have
\begin{equation}\label{cjdef2}
c_{j}(n) =\left(p^{j+1}-1-p^{j+1}\left\{\frac{n}{p^{j+1}}\right\}\right)\left(\frac{n}{p^{j+1}}-\left\{\frac{n}{p^{j+1}}\right\}\right)
\end{equation}
Where $j \ge 0$, $[x]$ is the floor function and $\{x\}$ is the fractional part of the rational number. 

\end{lem}
\begin{proof}
(1) Denote $t$ in base $p$ as $t=t_kp^k+t_{k-1}p^{k-1}+\dots+t_0$ and note that since $t\le n-1$, $t_k\le a_k$. 
We prove the result by induction on $j \ge 0$. 
For the base case $j=0$  whenever $t_0>a_0$ then that value of $t$ contributes 1 towards $c_0(n)$. 
The number of values that $t_0$ can take while being greater than $a_0$  is $p-1-a_0$. 
The number of values $t$ can take with $t_0>a_0$ for a fixed $t_0$ is $(n-a_0)/p$ and so we have 
\[
c_0(n)=(p-1-a_0)\frac{(n-a_0)}{p}=(p^{0+1}-1-a_0p^0)\left(\sum_{t=1}^ka_tp^{t-1}\right).
\]

For the  induction step we observe that 
a value $t$ contributes to $c_j(n)$ only if its $j$ smallest base $p$ digits satisfy:
 \begin{equation}\label{cond1}
 \sum_{v=0}^{j} t_vp^v>\sum_{v=0}^ja_vp^v.
 \end{equation}
 The number of last $j$ digits that would contribute to $c_j(n)$ is consequently
 \begin{equation}\label{cond2}
p^{j+1}-1-\sum_{v=0}^ja_vp^v.
  \end{equation}
  The number of $t$ that would satisfy \eqref{cond2} is therefore:
 \begin{equation}\label{cond3}
 \frac{n-\sum_{v=0}^ja_vp^v}{p^{j+1}}=\sum_{t=j+1}^ka_tp^{t-j-1}.
 \end{equation}
 And so we obtain the value of $c_j(n)$ by taking the product of\eqref{cond2} and \eqref{cond3}
 $$
 c_j(n)=\left((p^{j+1}-1)-\sum_{\uu=0}^ja_{\uu}p^{\uu}\right)\left(\sum^k_{t=j+1}a_tp^{t-j-1}\right).
$$

(2) The formula \eqref{cjdef2} follows by   rewriting the   sums in \eqref{cjdef1}, observing that 
\[
\sum_{t=j+1}^k a_tp^{t-j-1}=\left[\frac{n}{p^j}\right]= \frac{n}{p^j}-\left\{\frac{n}{p^j}\right\}
\] 
and
\[
p^j-1-\sum_{v=0}^j a_v p^v =p^j-1-p^j\left\{\frac{n}{p^j}\right\},
\]
as required.
\end{proof}

We apply Lemma \ref{cjdef} to prove our first formula for $\ord_p(\G_n)$.


\begin{proof}[Proof of Theorem \ref{opgn}.]

Substituting in \eqref{carries} the formula of  Lemma \ref{cjdef} yields
\begin{equation*}
\ord_p(\G_n) =\sum_{j=0}^{k-1} c_j(n)
 =\sum_{j=0}^{k-1}\left(\left((p^{j+1}-1)-\sum_{\uu=0}^ja_{\uu}p^{\uu}\right)\left(\sum^k_{t=j+1}a_tp^{t-j-1}\right)\right) 
\end{equation*}
Expanding the latter sum yields

\begin{equation}\label{ppart1}
\ord_p(\G_n)
=\sum_{j=0}^{k-1}\sum_{t=j+1}^ka_tp^t-\sum_{j=0}^{k-1}\sum_{t=j+1}^ka_tp^{t-j-1}-\sum_{j=0}^{k-1}\sum_{\uu=0}^{j}\sum_{t=j+1}^ka_{\uu}a_tp^{\uu+t-j-1}
\end{equation}
Now we simplify the three sums in  \eqref{ppart1} individually.
The first of these sums is 
\begin{equation}\label{ppart2}
\sum_{j=0}^{k-1}\sum_{t=j+1}^ka_tp^t =\sum_{j=1}^kja_jp^j.
\end{equation}
This sum extends to a $p$-adically convergent series,
which for $\alpha := \sum_{j=0}^{\infty} a_j p^j$ is
$f(\alpha) = \sum_{j=0}^{\infty} j a_j p^j$. In fact   $f: \ZZ_p \to \ZZ_p$
is not only a continuous function, but is   a $p$-adic  analytic function on $\ZZ_p$.

The second sum of \eqref{ppart1} can be re-written as:
\begin{equation}\label{ppart3}
\sum_{j=0}^{k-1}\sum_{t=j+1}^ka_tp^{t-j-1}
=\sum_{j=1}^ka_j\left(p^{j-1}+p^{j-2}+\dots+1\right)=\sum_{j=1}^ka_j\left(\frac{p^j-1}{p-1}\right)
\end{equation}

The third sum is a   bilinear sum, which satisfies the identity
\begin{equation}\label{ppart4}
\sum_{j=0}^{k-1}\sum_{b=0}^{j}\sum_{t=j+1}^ka_ba_tp^{b+t-j-1}=
\sum_{j=0}^k\frac{1}{p^{j+1}} \left(\sum_{\uu=0}^{j} a_{\uu}p^{\uu}\right) \left(\sum_{v=j+1}^ka_vp^{v}\right).
\end{equation}
By substituting  \eqref{ppart4}, \eqref{ppart2} and  \eqref{ppart3} into  \eqref{ppart1} we obtain the desired result.
\end{proof}

\begin{rem}\label{rem34}
 The total carry functions $c_j(n)$  seem of  interest in their own right.
 Bergelson and Leibman \cite{BL07}  study the class of all bounded functions 
 obtainable as finite  iterated combinations of the fractional part function $\{\cdot  \}$,
 calling them {\em generalized polynomials}. 
 They relate  generalized polynomials  to piecewise polynomial maps
on nilmanifolds and use this relation to  derive recurrence and distribution properties of 
values of such functions. The formula in Lemma \ref{cjdef}(2)
involves such functions. It  can be written as $c_j(n) = n g_{1,j}(n) + g_{2,j}(n)$
where $g_{i,j}(n)$ are the  bounded generalized polynomials  
$$
g_{1,j}(n) := 1 - \frac{1}{p^{j+1}} - \{\frac{n}{p^{j+1}}\},
$$
and 
$$
g_{2,j}(n) := - g_1(n) \big( p^{j+1}  \left\{\frac{n}{p^{j+1}}\right\} \big).
$$
each  of which is a  periodic function of $n$  with period $p^{j+1}$. 
However the function $c_j(n)$ 
itself is unbounded, so is not a generalized polynomial.
\end{rem}

%
%
%

\subsection{Prime-power divisibility of binomial  products $\G_n$: extreme values}\label{sec32}

Theorem \ref{opgn} permits an exact determination of the extreme behaviors of $\ord_p(\G_n)$.
We obtain a useful  upper bound on $\ord_p(\G_n)$ by retaining only those terms in Theorem \ref{opgn}
that are linear in the $a_i$, and this upper bound turns out to be sharp.

\begin{thm}\label{ordpgapk-1}
Let the prime $p$ be fixed.
 Then we have for all $n >0$ that
\begin{equation}\label{opgapk-1}
0 \le \ord_p(\G_n) \le M_p(n) := \sum_{j=0}^{k} ja_j p^j - \sum_{j=1}^{k} a_j\left(\frac{p^j-1}{p-1}\right).
\end{equation}
The equality 
$\ord_p(\G_n)=0$ 
holds if and only if $n = a p^k -1$,
with $1 \le a  \le p-1$.
The equality $\ord_p(\G_n)= M_p(n)$ holds if and only if $n=a p^k$ and
in that case
$$
\ord_p(\G_{ap^k})=a\left(kp^k-\frac{p^k-1}{p-1}\right).
$$
\end{thm}

\begin{proof}
The lower bound in \eqref{ordpgapk-1}  is immediate since  $\G_n$ is an integer.
The case of equality can be deduced directly from Kummer's theorem.
To have $\ord_p(\G_n) = 0$ using Kummer's theorem, all binomial
coefficients $\binom{n}{j}$ must be prime to $p$,
so there can be no value $0 \le j\le n$ such that subtracting $j$ from $n$
in base $p$ arithmetic results in  borrowing a digit. This fact requires that the  base $p$ digits $a_j$ of $n$  except
the top digit be $p-1$, i.e. $a_j= p-1$ for 
$0 \le j \le k-1$. There is no constraint on the top digit $a_k$ other than $a_k \ne 0$, since no borrowing
can occur in this digit.

The upper bound inequality $\ord_p(\G_n) \le M_p(n)$ follows immediately from Theorem \ref{opgn}.
The equality case will hold only if the bilinear term in that theorem vanishes; it is 
$$
\sum_{j=1}^{k} \frac{1}{p^j} \Big(\sum_{\uu=0}^{j}  a_{\uu}p^{\uu}\Big) \Big( \sum_{v=j+1}^ka_vp^{v}\Big).
$$
 This term will be positive whenever  two nonzero coefficients appear in the base $p$ expansion
of $n$, since one can find a nonzero cross term in this expression. Therefore equality can hold
only for those $n$ having one nonzero base $p$  digit, i.e. $n= a p^k$ with $a \ne 0$.
Direct substitution of $a_j=0$ for $0 \le j \le p-1$ yields the explicit formula for $\ord_p(\G_n)$ above.
\end{proof}

The previous result implies the following  bounds.
\begin{thm}\label{th29}
For each prime $p$, there holds for all $n \ge 1$, 
\begin{equation}\label{eq216}
0 \le \ord_p(\G_n) < n \log_p n.
\end{equation}
The value $n=p^k$ has  $\ord_p(\G_n) \ge n (\log_p n - 1)$.
\end{thm}

\begin{proof}
Only the upper bound in \eqref{eq216} needs to be verified. By Theorem \ref{opgn} we have
\[
\ord_p(\G_n)  < \sum_{j=1}^{k} ja_jp^j= kn-\sum_{\uu=0}^{k-1} (k-\uu)a_{\uu}p^{\uu} \le k n =  n \lfloor \log_p n \rfloor \le n \log_p n.
\]
For $n=p^k$ we have $\ord_p(\G_n) = kn - (1+ p + \cdots +p^{k-1}) \ge kn -n= n (\log_p n -1).$
\end{proof}

%
%

\section{ Formulas for $\ord_p(\G_n)$: Comparison and Implications}\label{sec7}

 We  presented three formulas for $\ord_p(\G_n)$ in Sections \ref{sec3}, \ref{sec4} and \ref{sec5}, respectively.
 Now we compare the formulas and discuss what they imply about features of the graph of $\W_2(\G_n) = \ord_2(\G_n)$ given
 in Figure \ref{fig21-ord2}
 and in the rescaled  Figure  \ref{fig22-ord2nlogn} following.

Each of the three formulas express $\ord_p(\G_n)$ as a difference of positive terms $\SP_{p, j}^{+}(n)$ and  $\SP_{p,j}^{-}(n)$,
for $j=1,2,3$.
The  term $\SP_{p, j}^{+}(n)$ is nondecreasing in $n$ in each formula, while  
the term $\SP_{p, j}^{-}(n)$ is smooth for $j=1$ but is oscillatory  in the other two formulas.\bigskip


%
%

\begin{minipage}{\linewidth}
\begin{center}
\begin{tabular}{| r ||r | r ||  r | r || r | r || c |}
\hline
 $n$ &  $\ord_2(D_n^{\ast})$  & $\ord_2(N_n^{\ast})$
  & ${\SP}_{2,2}^{+}(n)$  &  ${\SP}_{2,2}^{-}(n)$  & $\SP_{2,3}^{+}(n)$ & $\SP_{2,3}^{-}(n)$ & $\ord_2(\G_n)$  \\
\hline
$1$   &$0$ & $0$ & $0$  & $0$  &  $0$   &  $0$  &   $~0$\\
\hline
$2$   &$2$ & $1$ & $2$  & $1$  &  $2$   &  $1$ &   $~1$  \\
$3$   & $2$ & $2$  & $4$   &  $4$ &  $2$  & $2$  &  $~0$  \\
\hline
$4$   &10 &5 & $8$  & $3$  &  $8$   &  $3$  &   $~5$\\
$5$   &10 &8 &  $10$   &  $8$  &  $8$  & $6$  &  $~2$\\
$6$  &16 &12 &   $14$  &  $10$   &  $10$ & $6$  & $~4$  \\
$7$  &16 &16 &  $18$   &  $18$  &  $10$  & $10$  &  $~0$ \\
\hline
$8$  &40 &23 & $24$  & $7$  &  $24$   &  $7$   &   $17$\\
$9$  &40 &30 &   $26$   &  $16$  &  $24$  & $14$  & $10$ \\
$10$ &50 &38 &   $30$   &  $18$   &  $26$  & $14$  & $12$ \\
$11$   &50 &46 &   $34$ &   $30$ &$26$  &  $22$   & $~4$ \\
$12$  &74 &56 &  $40$   &  $22$ &  $32$  & $14$  &  $18$  \\
$13$  &74 &66 &  $44$   &  $36$  &  $32$  & $24$  &  $~8$ \\
$14$  &88 &77 &   $50$   &  $39$  &   $34$  & $23$  &$11$ \\
$15$  &88 &88 &   $56$   &  $56$ &  $34$  & $34$  & $~0$  \\
\hline
$16$   &$152$ & $103$ & $64$  & $15$  &  $64$   &  $15$  &   $49$\\

\hline
\end{tabular} \par
\bigskip
\hskip 0.5in {\rm TABLE 7.1.}  
{ Comparison of Theorems \ref{th31a},  \ref{th39} and \ref{opgn}  for $\ord_2(\G_n)$, $1 \le n \le 16$, divided in 
blocks $2^k \le n < 2^{k+1}$.}
\newline
\newline
\end{center}
\end{minipage}

  Table 7.1   presents numerical data on the three formulas for $p=2$ and small $n$, 
which illustrate their differences.
The lead term $\SP_{p,1}^{+}(n)= \ord_p(D_n^{\ast})$ in the first formula grows much more rapidly than  
the corresponding terms $\SP_{p, j}^{+}(n)$
for $j=2, 3$, 
evidenced by the asymptotic formula in Theorem \ref{th32a}.
The second and third formulas differ qualitatively in their second terms, 
in  that $\SP_{p,2}^{+}(n) := \frac{2}{p-1} S_p(n)$
 grows smoothly, being of size $\,\,n \log_p n + O(n)$, while $S_{p,3}^{+}(n)$ 
 grows less smoothly, having occasional jumps proportional to  $n \log n$ as the base $p$ digits $a_i$ vary.
 It appears  that the third formula gives the best upper bound of the three formulas, in the sense of minimal growth of  the positive term
 $S_{2, 3}^{+}(n)$ among the $S_{2, j}^{+}(n)$.
 In particular  $\SP_{2,2}^{+} (n) \ge S_{2,3}^{+}(n)$ holds in the table entries,  with equality holding for $n=2^k$.

We have shown that the function $\ord_p(\G_n)$ is of average size about $\frac{1}{2} n \log_p n$
and of size at most $n \log_p n$.  It is natural to make a companion plot to Figure \ref{fig21-ord2}
that rescales the values of $\ord_2(\G_n)$ by a factor $\frac{1}{2} n \log_2 n$, which we present in 
Figure \ref{fig22-ord2nlogn} below.
In the rescaled plot all values fall   in the interval $[0, 2]$, and have average size
around $1$, according to the discussion of $d_p(n)$ in feature (iii)  after Theorem \ref{th39}.


\begin{figure}[!htb]
\includegraphics[width=130mm]
 {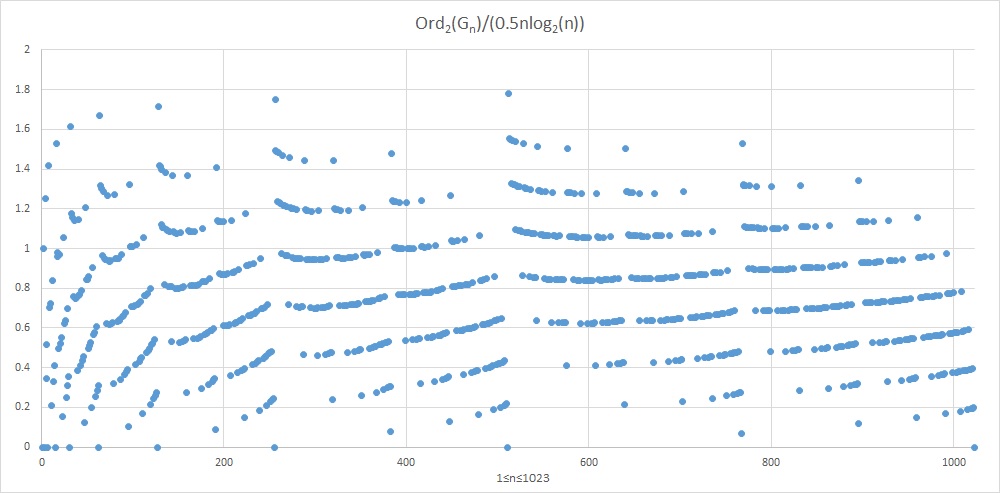}
\caption{Values of $\ord_2(\G_n)/(0.5 \, n \log_2 n)$, $1 \le n \le 1023=2^{10}-1$.}
\label{fig22-ord2nlogn}
\end{figure}

There are  several patterns visible  in 
  Figure \ref{fig21-ord2} 
  and in 
  Figure \ref{fig22-ord2nlogn}
above.
\begin{enumerate}
\item[(i)]
Between $n=2^k -1$ to  $n=2^k$ there is a  large jump visible in the value of $\ord_2(\G_n)$. 
The large  jump between $n=2^{k-1}$ and $n=2^k$ is 
quantified  more generally in Theorem \ref{ordpgapk-1}
for $\ord_p(G_n)$ between $n=p^{k} -1$ and $n= p^k$. 

\item[(ii)]
In Figure \ref{fig21-ord2} the values of $\ord_2(\G_n)$ 
between $n=2^{k}$ to  $n=2^{k+1}$ 
show a pattern  of diagonal lines 
or ``stripes".  These  diagonal lines are sloping upwards, have different lengths, and are roughly parallel to each other. 
The lengths of the ``stripes" vary in a predictable manner: from top to bottom, they first increase in length,
starting from the left, until they extend nearly  to the next level $2^{k+1}$,
remain stable at this width for  a while,  and thereafter decrease in length while
continuing to end near the next level $2^{k+1}$.
In the rescaled Figure \ref{fig22-ord2nlogn} these lines flatten out to give parallel ``stripes".

 Theorem \ref{th39} accounts for the ``stripes" visible in Figure \ref{fig22-ord2nlogn} via its term $-\frac{n-1}{p-1}\dgt_p(n)$,
 which shows that each stripe is  occupied by integers having a fixed value $d_p(n)= j$, in which  $d_p(n)=1$ labels the highest stripe,
 and each stripe downward increases $j$ by one.
 The odometer behavior of the function $d_p(n)$ also
 accounts for the horizontal width of the  ``stripes" and the motion of their  behavior over the interval $[p^k, p^{k+1}-1]$,
 and determines when they start near the left endpoint $p^k$ (small values of $d_p(n)$) or end near the 
 right endpoint $p^{k+1} -1$ (large values of $d_p(n)$). 
\item[(iii)]
On comparing ``stripes" in the interval 
 $n=2^k$ and $n=2^{k+1}-1$, with those at the next interval between $n=2^{k+1}$ and $n=2^{k+2}-1$, 
 the number of ``stripes"  increases by $1$.  
 This  increase in the number of stripes from the interval from $[p^k, p^{k+1}-1]$
 and the interval $[p^{k+1}, p^{k+2} -1]$ is accounted for by the allowed values of
 $d_p(n)$ labeling the given interval.
For $p=2$ there are exactly $k+1$ such stripes on the region $2^k \le n \le 2^{k+1} -1$,
there is an increase of exactly one stripe in the new interval,
and the spacing between the stripes is of width approximately $\frac{1}{k} \approx \frac{1}{\log_p n}$.
For a general prime $p$,  the increase in the number of stripes between the interval $[p^k, p^{k+1-1}]$
and $[p^{k+1}, p^{k+2}-1]$ is exactly $p-1$.

\item[(iv)]
The  values   between successive powers $2^k$ and $2^{k+1}$ have an apparent envelope
of largest growth. The envelope  appears to be of size proportional to $k 2^k$, a value approximately equal to $n \log_2 n$.
Under the rescaling by a factor proportional to $n \log n$, as is done in Figure \ref{fig22-ord2nlogn}, the ``stripes"  become
approximately flat, and they fall in the range $0 \le \ord_2(\G_n)/ (0.5\, n \log_2 n) \le 2.$

The envelope of largest growth for $d_p(n)$ is $n \log_p(n)-n$, as quantified in
Theorem \ref{th29}. On rescaling by a factor $1/ n \log_p n$, as done in Figure  \ref{fig22-ord2nlogn}
the  ``stripe"  of values with  $d_p(n)=j$ 
over the interval  $[p^n, p^{n+1} -1]$ becomes approximately flat.
Note that  the smooth main term $\frac{2}{p-1} S_p(n)$ in Theorem \ref{th39},
changes by at most $n$ over this  interval, which becomes after rescaling of  size
$O(\frac{1}{\log n})$ which is asymptotically 
negligible.
\item[(v)]
Inside the envelope of largest growth
are  irregular, correlated fluctuations in size consisting 
of  a structured set of  negative ``jumps" of various sizes. 
The  ``stripes" are a structure that indirectly emerges from correlations within the pattern of ``jumps".
We observe that the odometer behavior of $d_p(n)$ completely  accounts for the pattern of downward ``jumps" in the values of $\ord_p(\G_n)$,
The odometer  behavior also produces the self-similar structure in the  
locations of values $n$ with $d_p(n)=j$ which produce the ``stripes". 
\end{enumerate}

 A remaining
 mystery of the functions $\ord_p(\G_n)$ concerns the interpretation of their generalizations $\W_b(\G_n)$ introduced for all integers $b \ge 2$ 
 in Section \ref{sec52}  using base $b$ radix expansions.
One can view  $\G_n$  as a universal object that encodes all the
data $\W_p(n)= \ord_p(\G_n)$ for prime $p$. If so, what might be the universal
object encoding the data for all $\W_b(\G_n)$, $b \ge 1$?


\section{Binomial Products and Prime Counting Estimates}\label{sec7b}

Binomial coefficients are well known to encode 
 information about the distribution
of prime numbers. 
This holds more generally for { integer factorial ratios},
which are one-parameter families of ratios of products of factorials that are integers
for all parameter values. Let $\pi(x)$ count the number of primes $p \le x$.
Chebyshev \cite{Cheb1852} 
 used the  integer factorial ratios   $A_n :=\frac{(30n)! n!} {(15n)! (10n)! (6n)!}$
to obtain his bounds 
$$
0.92 \,\frac{x}{\log x} \le \pi(x) \le 1.11 \frac{x}{\log x},
$$
It  is known that the in principle the ensemble of integer factorial ratios contain
enough information  to  give a proof
of the prime number theorem, see Diamond and Erd\H{o}s \cite{DE80}. 
However their  method to show the existence of a suitable sequence of such ratios
 used the prime number theorem as an input, so did not yield
 an elementary proof of the prime number theorem, see the discussion in Diamond \cite[Sect. 9]{Di82}.

We now consider the restricted set of all products of binomial coefficients.
The asymptotic formula for $\log(\G_n)$ in Theorem \ref{th22b} gives
$\log(\G_n) = \frac{1}{2} n^2 + O( n \log n)$. 
The prime number theorem says that $\pi(n) \sim \frac{n}{\log n}$ and is therefore
equivalent to the assertion
\begin{equation}\label{PNT1}
\log (\G_n) = \frac{1}{2} \pi(n) \,n \log n + o(n^2).
\end{equation}
The  individual binomial products $\G_n$ 
imply   Chebyshev-type bounds for
$\pi(x)$. 
The product formula expressing unique factorization (\cite{AW45}) yields on
taking a logarithm
\begin{equation} \label{upper}
\log (\G_n) = \sum_{p \le n} \ord_p(\G_n)\log p.
\end{equation}
The left side is estimated by the  asymptotic formula for $\log(G_n)$ in Theorem \ref{th22b},
$$
\log(\G_n) = \frac{1}{2} n^2 - \frac{1}{2} n \log n +(1- \frac{1}{2} \log (2 \pi))n + O(\log n).
$$
Since $1- \frac{1}{2} \log (2 \pi)>0$, 
one has for sufficiently large $n$, and in fact for all $n \ge 3$, 

\begin{equation}\label{lower}
\log(\G_n) \ge  \frac{1}{2} n^2 - \frac{1}{2} n \log n.
\end{equation}

Theorem \ref{th29} states  $\ord_p(\G_n) \le n \log_p(n)$, which upper bounds the right side of \eqref{PNT1} by
 \begin{equation}\label{upper}
\log(\G_n) = \sum_{p \le n} \ord_p(\G_n)\log p \le \sum_{p \le n} (n \log_p n) \log p =  \pi(n)\, n \log n.
\end{equation}
Combining these two inequalities yields the Chebyshev-type lower bound, 
$$
\pi(n) \ge \frac{1}{2} (\frac{n}{\log n}) - \frac{1}{2},
$$
 valid  for  all $n \ge 2$. This bound loses  a constant factor of $2$ compared to the prime number theorem,
so is much worse than Chebyshev's lower bound. But it  has a  redeeming feature:
{\em in an average sense most  $\ord_p(\G_n)$ are of size near $\frac{1}{2} n \log_p n.$}
This  observation follows from Theorem \ref{th39}, combined with  the fact that $d_p(n)$ has mean  $\frac{p-1}{2}\log_p n$,
provided  that one averages over $n$.  
This (heuristic) observation would save  back  exactly the  factor of $2$ lost in the argument above on the right side of \eqref{upper},
which therefore suggests the possibility of  an approach to proving
 the  prime number theorem\footnote{The lower  bound $\pi(x) \ge \frac{x}{\log x} + o(\frac{x}{\log x})$ is  known to be equivalent
to  the full prime number theorem} via radix expansions.

There remain serious obstacles to obtaining  a proof of the prime number theorem along these lines. If 
one holds $p$ fixed and varies $n$, then one can rigorously show
that $\ord_p(\G_n)$ is usually of size near $\frac{1}{2} n \log_p n$.
However the  sum $ \sum_{p \le n} \ord_p(\G_n)\log p$ appearing  in \eqref{upper} makes a different averaging:
{\em holding  $n$ fixed and letting $p$ vary, restricting to  $p \le n$.}  The analysis of this new averaging
leads to new kinds of  arithmetical sums involving radix expansions
and we  leave their investigation  to future work.

Another relation of binomial products $\G_n$ to
the distribution of prime numbers arises via their connection to products
of Farey fractions. This connection via M\"{o}bius inversion creates other products
of binomial coefficients which may be directly related to the Riemann hypothesis,
for which see \cite{LM14r}.
Moreover, there are  general relations known between families of
integer factorial ratios and  the Riemann hypothesis.
For  some recent  work on their structure,  see  Bell and Bober \cite{BB09}
and Bober \cite{Bo09}.

%
%
%
\section{Appendix A: Asymptotic expansions for $\log(D_n^{\ast}), \log(N_n^{\ast})$ and $\log(\G_n$)}\label{secA}

In this appendix we derive full asymptotic expansions for 
the logarithms of the {\em superfactorial function} $N_n^{\ast}=\prod_{k=1}^n  k!$, the {\em hyperfactorial function}
$D_n^{\ast}=\prod_{k=1}^n  k^k$
 and the binomial products $\G_n= D_n^{\ast}/ N_n^{\ast}$.

We start from the formulas
\begin{eqnarray}\label{200c}
N_n^{\ast} &= &\Gamma(n+1) G(n+1), \\
&&\nonumber \\
D_n^{\ast} &= & \frac{ \Gamma(n+1)^n}{G(n+1)}, \label{200d}
\end{eqnarray}
in which 
$\Gamma(n)$ denotes the Gamma function
and  $G(n)$ denotes  the Barnes $G$-function, both discussed below.
The formulas \eqref{200c} and \eqref{200d} follow from the standard
identities $\Gamma(n+1) = n!$ and $G(n+1) = 1! 2! \cdots (n-1)!$, respectively.
These two formulas yield
\begin{equation}\label{gndef2}
\G_n =\frac{D_n^{\ast}}{N_n^{\ast}}=\frac{\Gamma(n+1)^{n-1}}{G(n+1)^2}.
\end{equation}

The {\em Gamma function} $\Gamma(z)$ was originally defined to interpolate the factorial function,
and was studied at length by Euler (see \cite[Sect. 2.3]{Lag13}, and Artin \cite{Art64}).
It  satisfies a functional equation
$\Gamma(z+1) = z\Gamma(z)$, and has $\Gamma(1)=1$, which yields $\Gamma (n+1) = n!.$
Its reciprocal is an entire function of order $1$ (and maximal type) defined by the everywhere convergent Hadamard product
$$
\frac{1}{\Gamma(z)} = e^{\gamma z} z \prod_{k=1}^{\infty} \big( 1 + \frac{z}{k}\big) e^{- \frac{z}{k}}, 
$$
in which $\gamma \approx 0.57721$ denotes Euler's constant.
The  asymptotic expansion of the logarithm of the  Gamma function 
is related to Stirling's formula. It was determined by Stieltjes, who gave a precise notion of
asymptotic expansion (see \cite[Sect. 3.1, and (3.1.10)]{Lag13},  \cite[Chapter 5, (5.11.1)]{NIST}). It states\footnote{The
  asymptotic expansion of $\log \Gamma(z)$ is extremely similar  to that of 
 $\log \Gamma(z+1)$,  changing only one   term $\frac{1}{2}\log z$ to $- \frac{1}{2} \log z$,
via the identity $\log \Gamma(z+1 )= \log \Gamma(z) + \log z$.
The expansion of
$\log \Gamma(z)$ given in
Whittaker and Watson \cite[Sect. 12.33]{WW63} uses an older notation for Bernoulii numbers: their notation  $B_{k}$ corresponds
to $ |B_{2k}|$ in our notation.}
 for any fixed $N \ge 1$ that
\begin{equation}
\begin{aligned}
&\log \Gamma(z+1)=
z\log z -z + \frac{1}{2} \log z + \frac{1}{2} \log (2 \pi) \\&\quad\quad\quad + \sum_{k=1}^N 
 \frac{B_{2k}}{2k(2k-1)} \frac{1}{z^{2k-1}}+ O \big( \frac{1}{z^{2N+1}}\big),
\end{aligned}
\end{equation}
where $B_k$ denote the Bernoulli numbers, as 
determined by the generating function $ \frac{t}{e^t-1} = \sum_{k=0}^{\infty} B_k  \frac{t^k}{k!},$
in particular $B_1=-\frac{1}{2}$.
This formula   is known to be valid in any sector $- \pi + \epsilon < Arg(z) \le \pi - \epsilon$ of the complex plane, 
with the implied  $O$-constant depending on both  $N$ and $\epsilon$.

The  {\em Barnes $G$-function} was introduced  by Barnes \cite{Bar1900} in 1900.
It is a less well known than the Gamma function, 
and  is closely related to a generalization of the gamma function, the  double gamma function,
also introduced by Barnes  (\cite{Bar1899}, \cite{Bar1901}), see also \cite[Sect. 5.17]{NIST}.
It satisfies the functional equation
$$
G(z+1) = \Gamma(z) G(z), 
$$
and has $G(1)=1$, which yields $G(n+1) =(n-1)! (n-2)!\cdots 1!,$ and  also $G(n+2)= N_n^{\ast}$.
Recently the Barnes $G$-function has assumed  prominence
from its appearance in  formulas  relating  the Riemann
zeta function to random matrix theory.
These formulas appear  in random matrix theory for the Circular Unitary Ensemble,
and in conjectured formulas for  moments of the Riemann zeta function
on the critical line $Re(s)= \frac{1}{2}$, 
see Keating and Snaith \cite{KS00} and Hughes \cite{Hug03}.

The Barnes $G$-function 
is an entire function of order $2$ defined by the everywhere convergent  Weierstrass product
$$
G(z) := (2 \pi)^{\frac{z}{2}} \exp \big( -\frac{1}{2} ( z + z^2( 1 + \gamma))\big)\prod_{k=1}^{\infty} ( 1 + \frac{z}{k})^k \exp( \frac{z^2}{2k} - z ),
$$
where again $\gamma$ is Euler's constant.
The  asymptotic expansion  for  the Barnes $G$-function\footnote{Barnes \cite{Bar1900} follows a different  convention for
Bernoulli numbers: his $B_k$ corresponds to  $|B_{2k}|$ in the notation used here.
We have altered his formula accordingly.} (\cite[p. 285]{Bar1900}) has the form,
for any fixed $N \ge 1$, 
\begin{equation}\begin{aligned}
\log G(z+1) &= \frac{1}{2} \,z^2 \log z-\frac{3}{4}z^2 + \big( \frac{1}{2} \log(2\pi)\big) z-\frac{1}{12}\log z\\
&\quad +(\frac{1}{12}-\log A) + \sum_{k=1}^N
\frac{B_{2k+2}}{2k(2k+2)}\frac{1}{z^{2k}}+O\left(\frac{1}{z^{2N+2}}\right),
\end{aligned}\end{equation}
where $A=\exp \big( \frac{1}{12}- \zeta^{'}(-1)\big)$ is the Glasher-Kinkelin constant (Kinkelin \cite{Kin1860}, Glaisher \cite{Gla1878}, \cite{Gla1894}),  which
has a numerical value of $A\approx 1.2824271291\dots$ and
where  $B_k$  denote  the Bernoulli numbers. This asymptotic expansion is valid in any sector 
$- \pi + \epsilon < Arg(z) \le \pi - \epsilon$ of the complex plane, 
with the implied  $O$-constant depending on both  $N$ and $\epsilon$.
The original derivation of Barnes did not control the error term but Ferreira and L\"{o}pez \cite{FL01}
later obtained an asymptotic expansion\footnote{ Their expansion contains a term $z \log \Gamma(z+1)$
so the asymptotic expansion of $\log \Gamma(z+1)$ must be substituted
in their formula.}  with error term as above, see also  Ferreira \cite{Fe04} and Nemes \cite{Nem14}.

The asymptotic expansions  for both $\Gamma(z+1)$ and $G(z+1)$ when extended to all orders 
are  divergent series. That is, the  associated power series in  $w= \frac{1}{z}$, taking $N = \infty$, has radius of
convergence zero, a fact which follows from the super exponential  growth of the even Bernoulli numbers $|B_{2n}| \sim 4 \sqrt{ \pi n} (\frac{n}{\pi e})^{2n}$
as $n \to \infty$. 
The derived asymptotic expansions for $\log (D_n^{\ast})$, $\log(N_n^{\ast})$ and 
and $\log(\G_n)$ below also involve Bernoulli numbers and are also divergent series when extended to all orders.

We now state asymptotic expansions for $\log(D_n^{\ast})$ and $\log(N_n^{\ast})$.

\begin{thm}\label{th31b}
(1) The superfactorial $N_n^{\ast} = \prod_{k=1}^n k!$ has  an  asymptotic expansion for
$\log(N_n^{\ast})$  valid to any given order $N \ge 1$, valid uniformly
for all $n \ge 2$, of the form 
\begin{eqnarray*}\label{superfac}
\log  (N_n^{\ast}) &=& \frac{1}{2} n^2 \log n - \frac{3}{4}n^2 + n \log n + \big(\frac{1}{2}\log(2 \pi)  -1) \big) n + \frac{5}{12} \log n +\\
&& \quad\quad + \, c_0   + \sum_{j=1}^N  c_j (\frac{1}{n^j}) \,+ O\left(\frac{1}{n^{N+1}}\right).
\end{eqnarray*}
The constant $c_0= \frac{1}{2}\log(2 \pi) + \frac{1}{12} - \log A$
where $A=\exp \big( \frac{1}{12}- \zeta^{'}(-1)\big)$ is the Glaisher-Kinkelin constant, for $j \ge 1$ the coefficients
$c_j$ are explicitly computable  rational numbers,and the implied $O$-constant depends on $N$.

(2) 
The hyperfactorial $D_n^{\ast} = \prod_{k=1}^n k^k$ has  an  asymptotic expansion
for $\log (D_n^{\ast})$  up to any given order $N \ge 1$, valid uniformly for all $n \ge 2$, 
of the form
\begin{eqnarray*}\label{hyperfac}
\log  (D_n^{\ast}) &=& \frac{1}{2} n^2 \log n - \frac{1}{4}n^2  + \frac{1}{2} n \log n 
 + \frac{1}{12} \log n +\quad\quad\quad\quad\\ 
&&\quad\quad + \, d_0+
\sum_{j=1}^N  d_j \big( \frac{1}{n^j}\big) +
O\left(\frac{1}{n^{N+1}}\right).
\end{eqnarray*}
The constant
$d_0= \log A$,  where $A$ is the Glaisher-Kinkelin constant, 
 for $j \ge 1$ the coefficients $d_j$ are explicitly computable rational numbers, and the implied $O$-constant depends on $N$.
\end{thm}

\begin{proof}
[Proof of Theorem \ref{th31a}]
(1) We have
$$\log(N_n^{\ast}) = \log \Gamma(n+1) + \log G(n+1).$$
Substituting the asymptotic series for $\Gamma(n+1)$ and $G(n+1)$ each term by term yields \eqref{superfac}.
Here for $k \ge 1$ we have  $c_{2k}= \frac{B_{2k+2}}{2k(2k+2)}$ while $c_{2k-1}= \frac{B_{2k}}{(2k)(2k-1)}.$

(2) We have
$$
\log(D_n^{\ast}) = n \log \Gamma(n+1) - \log G(n+1).
$$
Substituting the asymptotic series $\Gamma(n+1)$ and $G(n+1)$  term by term on the right side yields \eqref{hyperfac}.
Multiplying by $n$ in the first term on the right shifts the  coefficient indices of the asymptotic expansion of $\log \Gamma(n+1)$ 
down by $1$.
For $k \ge 1$ we have  
$$
d_{2k} = c_{2k+1}-c_{2k}= \frac{B_{2k+2}}{(2k+2)(2k+1)}-\frac{B_{2k+2}}{2k(2k+2)}= \frac{B_{2k+2}}{2k(2k+1)(2k+2)}
$$
while $d_{2k-1}= 0.$
\end{proof}

Theorem \ref{th31a} immediately yields asymptotic expansion for $\log(\G_n)$.
The resulting asymptotic behavior of  $\log(\G_n)$ is of smaller order of magnitude, since the
leading  term in the asymptotic series of  $\log(N_n^{\ast})$ and $\log (D_n^{\ast})$ 
on the right side of \eqref{cancel} cancel.

\begin{thm}\label{th22b}
The  complete binomial products $\G_n=\prod_{j=1}^n {{n}\choose{j}}$ have an asymptotic expansion for  $\log(\G_n)$ 
to any given order $N \ge 1$, valid uniformly for all $n \ge 2$, of the form
\begin{equation}
\notag
\log  (\G_n)=\frac{1}{2}n^2-\frac{1}{2}n\log n+ \big(1- \frac{1}{2} \log (2 \pi))\big)\, n -\frac{1}{3}\log n+g_0+\sum_{j=1}^N g_j \big( \frac{1}{n^j}\big)  + 
O\left(\frac{1}{n^{N+1}}\right).
\end{equation}
Here $g_0=  -\frac{1}{2} \log (2 \pi)  - \frac{1}{12}+ 2 \log A$ where $A$ is the Glaisher-Kinkelin constant,
 for $j \ge 1$  the coefficients $g_j$  are explicitly computable rational numbers, and the implied $O$-constant depends on $N$.
\end{thm}

\begin{proof}
We have 
\begin{equation}\label{cancel}
\log (\G_n) = \log (D_n^{\ast}) - \log (N_n^{\ast}).
\end{equation}
Direct substation from Theorem \ref{th31a} then gives the result.
Here the terms $g_j$ for $j \ge 1$ are given by $g_j= d_j - c_{j}$, so are 
$$
g_{2k} = d_{2k} - c_{2k} = -\frac{B_{2k+2}}{2k(2k+1)(2k+2)} - \frac{B_{2k+2}}{2k(2k+2)}= -\frac{B_{2k+2}}{2k(2k+1)},
$$
while
$g_{2k-1}= -c_{2k-1} = -\frac{B_{2k}}{2k(2k-1)}$.
\end{proof}

 Table A.1 below gives coefficients of the first few terms $\frac{1}{n^k}$ in the asymptotic expansions above
 in Section 3.
The small size of the coefficients  in the table is quite misleading; later coefficients become very large, 
since the Bernoulli numbers satisfy $|B_{2n}| \sim 4 \sqrt{ \pi n} (\frac{n}{\pi e})^{2n}$ as $n \to \infty$. \medskip

%
%

\begin{minipage}{\linewidth}
\begin{center}
\begin{tabular}{| c || c | c | c | c | c | c |}
\hline
$\mbox{Coefficient}$ & $k=1$ & $k=2$  &  $k=3$  & $k=4$ & $k=5$  & $k=6$\\
\hline
~&~&~&~&~&~&~\\
$\log \Gamma(z+1)$  &   $\frac{1}{12}$  & $0$  &  $-\frac{1}{360}$   &  $0$  & $\frac{1}{1260}$  & $0$\\
~&~&~&~&~&~&~\\
$\log G(z+1)$  &   $0$  & $-\frac{1}{240}$  &  $0$  &  $\frac{1}{1008}$  & $0$  & $-\frac{1}{1440}$\\
~&~&~&~&~&~&~\\
$c_k$ &   $\frac{1}{12}$  & $-\frac{1}{240}$  &  $-\frac{1}{360}$   &  $\frac{1}{1008}$  & $\frac{1}{1260}$  & $-\frac{1}{1440}$\\
~&~&~&~&~&~&~\\
$d_k$ &   $0$  & $\frac{1}{720}$  &  $0$  &  -$\frac{1}{5040}$  & $0$  & $\frac{1}{10080}$\\
~&~&~&~&~&~&~\\
$g_k$ &   $-\frac{1}{12}$  & $\frac{1}{180}$  &  $\frac{1}{360}$  &  $-\frac{1}{860}$  & $-\frac{1}{1260}$   & $\frac{1}{1260}$ \\  
~&~&~&~&~&~&~\\
\hline
\end{tabular} \par
\bigskip
\hskip 0.5in {\rm TABLE A.1.}  
{\em Asymptotic expansion coefficients $c_k, d_k, g_k$.}
\newline
\newline
\end{center}
\end{minipage}

\begin{rem}\label{rem93}
(1) We may extend $\G_n$ to 
an analytic function $\G^{an}(z)$ of a complex variable $z$ on the complex plane cut along the nonpositive real axis, 
 using the right side of \eqref{gndef2} as a definition:
$$
\G^{an}(z) := \frac{\Gamma(z+1)^{z-1}}{G(z+1)^2}.
$$
Since the Gamma function has no zeros, the function $\log \Gamma(z+1)$
is well-defined on the cut plane,  
and we may set $\Gamma(z+1)^{z-1} := \exp ( (z-1) \log \Gamma(z+1))$, choosing that branch of the logarithm that is 
real on the positive real axis.  The function $\G^{an}(z)$ is not a meromorphic function; instead, it analytically continues to
a multi-valued function on a suitable Riemann surface which covers the complex plane punctured at the negative integers.
It  is an example of an ``endlessly continuable" function, as discussed in Sternin and Shalatov \cite{SS96} or  Sauzin \cite{Sau13}.

(2) For number-theoretic applications (as in \cite{LM14r}) 
one extends the values $\G_n$  to  positive real $x$ another way,  making it a step function
setting  $\G_x:=\G_{ \lfloor x\rfloor}$.
For the step function definition the  function $\log (\G_x)$,  viewed as  a function of a real variable $x$,  has jumps of size $\gg n$
at integer values of $x$. These jumps are of much larger size than most terms in the asymptotic expansion of Theorem \ref{th22b}.
In this case the  asymptotic expansion in Theorem \ref{th22b} 
is valid to all orders  $\frac{1}{n^k}$  exactly  at {\em integer} points $x=n$.

\end{rem}
%
%
%
\section{Appendix B: Equality of $\W_b(\G_n)$ and $\W_b^{\ast}(\G_n)$}\label{secB}

In this  Appendix we shows that the functions $\W_b(\G_n)$ and $\W_b^{\ast}(\G_n)$ 
introduced in Section \ref{sec52} and \ref{sec31} are equal. 
Recall from Section \ref{sec52} that
$$
\W_b(\G_n) := \frac{1}{b-1} \Big(2\Sum_b(n) - (n-1) \dgt_b(n)   \Big).
$$
Recall from  Section \ref{sec31} that
$$
\W_b^{\ast}(\G_n) := \sum_{j=1}^kja_jp^j-\left(\sum_{j=1}^ka_j\left(\frac{p^j-1}{p-1}\right)+
\sum_{j=0}^k\frac{1}{p^{j+1}}\Big(\sum_{\uu=0}^ja_{\ell}p^{\uu}\Big)\Big( \sum_{v=j+1}^ka_vp^{v}\Big)\right).
$$
Both these functions are  expressed  in terms of the base $b$ expansion $n = \sum_{j=0}^k a_j b^j$.
These functions were defined so that  they satisfy 
$\W_p(\G_n)= \W_p^{\ast}(\G_n) = \ord_p(\G_n)$ for  $ p$ a prime.

\begin{thm}\label{thB1}
For each integer  $b \ge 2$ there holds
$$
\W_b(\G_n) = \W_b^{\ast}(\G_n) \quad \mbox{for all} \quad n \ge 1.
$$
\end{thm}

\begin{proof}
We re-express both functions using the floor function. 
We first have
\begin{equation*}\begin{aligned}
\W_b(\G_n) 
&= \frac{1}{b-1}\left(2\sum_{m=1}^{n-1}\left(m-(b-1)\sum_{j\ge 1}\lfloor{\frac{m}{b^j}\rfloor}\right)-(n-1)\sum_{j\ge 0}\left(\lfloor{\frac{n}{b^j}\rfloor}-b\lfloor{\frac{n}{p^{j+1}}\rfloor}\right)\right)\\
&=(n-1)\sum_{j\ge 1}\lfloor{\frac{n}{b^j}\rfloor}-2\sum_{m=1}^{n-1} \Big( \sum_{j\ge1}\lfloor{\frac{m}{b^j}\rfloor}\Big).\\
\end{aligned}\end{equation*}
We also have 
\begin{equation*}\begin{aligned}
 \W_b^{\ast}(\G_n) &=
\sum_{j=1}^k jb^j\left(\lfloor{\frac{n}{b^j}\rfloor}-b\lfloor{\frac{n}{b^{j+1}}\rfloor}\right)
-\left(\sum_{j=1}^ka_j\left(\frac{b^j-1}{b-1}\right)+
\sum_{j=0}^k\frac{1}{b^{j+1}}\Big(\sum_{\uu=0}^ja_{\uu}p^{\uu}\Big)\Big( \sum_{v=j+1}^ka_vb^{v}\Big)\right)\\
&=\sum_{j\ge1}b^j\lfloor{\frac{n}{b^j}\rfloor}-\left(\sum_{j=1}^k\left(\lfloor{\frac{n}{b^j}\rfloor}-b\lfloor{\frac{n}{b^{j+1}}\rfloor}\right)\left(\frac{b^j-1}{b-1}\right)+
\sum_{j=0}^k\frac{1}{b^{j+1}}\Big(\sum_{\uu=0}^ja_{\uu}b^{\uu}\Big)\Big( \sum_{v=j+1}^ka_v b^{v}\Big)\right)\\
&=\sum_{j\ge1}b^j\lfloor{\frac{n}{b^j}\rfloor}-\left(\sum_{j\ge1}\lfloor{\frac{n}{b^j}\rfloor}+
\sum_{j=0}^k\frac{1}{b^{j+1}}\Big(\sum_{\uu=0}^ja_{\uu}b^{\uu}\Big)\Big( \sum_{v=j+1}^ka_vb^{v}\Big)\right)\\
&=\sum_{j\ge1}b^j\lfloor{\frac{n}{b^j}\rfloor}-\left(\sum_{j\ge1}\lfloor{\frac{n}{b^j}\rfloor}+
\sum_{j\ge1}\left(\frac{n}{b^j}-\lfloor{\frac{n}{b^j}\rfloor}\right)\left(b^j\lfloor{\frac{n}{b^j}\rfloor}\right)\right).\\
\end{aligned}
\end{equation*}
Combining these two formulas yields
  \begin{equation*}\begin{aligned}
\W_b(\G_n)-  \W_b^{\ast}(\G_n) &=
n\sum_{j\ge 1}\Big\lfloor{\frac{n}{b^j}\Big\rfloor}-\sum_{m=1}^{n-1}\Big(\sum_{j\ge 1}\Big\lfloor{\frac{m}{b^j}\Big\rfloor}\Big)
-\sum_{j\ge 1}b^j \left(\sum_{k=1}^{\lfloor{n/b^j}\rfloor}k\right)\\
&= \sum_{j\ge 1}\left(n\Big\lfloor{\frac{n}{b^j}\Big\rfloor}-\sum_{m=1}^{n-1}\Big\lfloor{\frac{m}{b^j}\Big\rfloor}
-b^j\left(\sum_{k=1}^{\lfloor{n/b^j\rfloor}}k\right)\right).\\
\end{aligned}
\end{equation*}
We  assert that each inner sum (for fixed $j$) on the right side of this sum is $0$. 
To see this, we have
\begin{eqnarray*}
n\Big\lfloor{\frac{n}{b^j}\Big\rfloor}-\sum_{m=1}^{n-1}\Big\lfloor{\frac{m}{b^j}\Big\rfloor}
-b^j\left(\sum_{k=1}^{\lfloor{n/b^j\rfloor}}k\right) &=&
\sum_{k=1}^{\lfloor{n/b^j\rfloor}}(n-b^j k)-\sum_{m=1}^{n-1}\lfloor{\frac{m}{b^j}\rfloor}.\\
&=&\sum_{k=1}^{\lfloor{n/b^j\rfloor}}(n-b^j k)-\sum_{m=1}^{n-1}\lfloor{\frac{n-m}{b^j}}\rfloor =0.
\end{eqnarray*}
The equality to zero  on the last line follows by a counting argument. 
We evaluate the last sum on the right in blocks of length $b^j$, taking $1 + (k-1)b^j \le m \le kb^j$ for $1 \le k \le \lfloor n/b^j \rfloor$.
If $n = b_j \lfloor n/b^j \rfloor+ a$ then the first block contributes $b^j \lfloor n/b_j \rfloor -(b^j -a) = n - b^j$. The
$k$-th block contributes $n - b^jk$ similarly, and a possible final ``short" block  contributes $0$. 
\end{proof}


\subsection*{Acknowledgments}
We thank J. Arias de Reyna for supplying the plot in Fig. 2 and for several corrections.
We thank the reviewer for helpful comments. 
The first author is indebted to Harm Derksen 
for raising   the topic of products  of Farey fractions , see \cite{DL11a}, \cite{DL11b}.
The  work  of the second author began as part of an REU project at the University of Michigan 
in 2013 with the first author as mentor. Work of the first author was
partially supported by NSF grants DMS-1101373 and DMS-1401224.

%
%
%


\end{document}